\documentclass [12pt]{amsart}
\usepackage[utf8]{inputenc}
\pdfoutput=1
\usepackage{amsmath,amssymb,amsthm,amsfonts}
\usepackage{mathtools}%

\usepackage[english]{babel}%
\usepackage{comment}%
\usepackage{hyperref}%
\usepackage{bbm}%
\usepackage{bm}%
\usepackage{mathrsfs}%

\usepackage{tikz,graphicx,color}
\usepackage{tikz-cd}%
\usepackage{tikz-3dplot}%
\usetikzlibrary{calc}%
\usetikzlibrary{arrows}%
\usetikzlibrary{shapes}%
\usetikzlibrary{patterns}%
\usetikzlibrary{positioning}%
\usetikzlibrary{arrows.meta}
\usetikzlibrary{decorations.markings}
\usetikzlibrary{knots}

\usepackage{epstopdf}%

\usepackage[arrow]{xy}%
\usepackage{diagbox}%
\usepackage{subfig}%
\usepackage{arcs}%
\usepackage{xcolor}%
\usepackage{xspace}

\usepackage[margin=1.0in]{geometry}%

\usepackage{enumitem}
\usepackage{letltxmacro}
\usepackage{thmtools,etoolbox}

\def\myarabic#1{\normalfont(\roman{#1})}
\newlist{theoremlist}{enumerate}{1}
\setlist[theoremlist]{label=\myarabic{theoremlisti},ref={\myarabic{theoremlisti}},itemindent=0pt,labelindent=0pt,
  leftmargin=*,noitemsep}

\makeatletter
\renewcommand{\p@theoremlisti}{\perh@ps{\thetheorem}}
\protected\def\perh@ps#1#2{\textup{#1#2}}
\newcommand{\itemrefperh@ps}[2]{\textup{#2}}
\newcommand{\itemref}[1]{\begingroup\let\perh@ps\itemrefperh@ps\ref{#1}\endgroup}
\newcommand{\itemreff}[1]{\begingroup\let\perh@ps\itemrefperh@ps\getrefnumber{#1}\endgroup}
\makeatother

\usepackage{nameref,hyperref}
\usepackage[capitalize]{cleveref}

\newtheorem{theorem}{Theorem}[section]

\newtheorem{lemma}[theorem]{Lemma}
\newtheorem{proposition}[theorem]{Proposition}
\newtheorem{corollary}[theorem]{Corollary}
\newtheorem{conjecture}[theorem]{Conjecture}
\theoremstyle{definition}
\newtheorem{remark}[theorem]{Remark}
\theoremstyle{definition}
\newtheorem{notation}[theorem]{Notation}
\theoremstyle{definition}
\newtheorem{definition}[theorem]{Definition}

\theoremstyle{definition}
\newtheorem{problem}[theorem]{Problem}
\theoremstyle{definition}
\newtheorem{example}[theorem]{Example}

\crefname{figure}{Figure}{Figures}

\def\figref#1(#2){Figure~\hyperref[#1]{\ref*{#1}(#2)}}

\addtotheorempostheadhook[theorem]{\crefalias{theoremlisti}{theorem}}
\addtotheorempostheadhook[lemma]{\crefalias{theoremlisti}{lemma}}
\addtotheorempostheadhook[proposition]{\crefalias{theoremlisti}{proposition}}
\addtotheorempostheadhook[corollary]{\crefalias{theoremlisti}{corollary}}
\addtotheorempostheadhook[definition]{\crefalias{theoremlisti}{definition}}
\addtotheorempostheadhook[conjecture]{\crefalias{theoremlisti}{conjecture}}

\def\Bcal{\mathcal{B}}\def\Ocal{\mathcal{O}}\def\Pcal{\mathcal{P}}\def\Rcal{\mathcal{R}}\def\Xcal{\mathcal{X}}

\def\C{{\mathbb{C}}}
\def\R{{\mathbb{R}}}

\def\Z{{\mathbb{Z}}}

\def\F{{\mathbb{F}}}

\newcommand\parr[1]{{({#1})}}
\def\<{{\langle}}
\def\>{{\rangle}}
\def\eps{{\epsilon}}
\def\la{{\lambda}}

\def\id{\operatorname{id}}

\def\maxx{\operatorname{max}}
\def\minn{\operatorname{min}}

\def\Gr{\operatorname{Gr}}

\def\Pio{\Pi^\circ}
\def\Povtp_#1{\Pi_{#1}^{>0}}
\def\Povtnn_#1{\Pi_{#1}^{\geq0}}
\def\BND{\Bcal}

\def\id{{\operatorname{id}}}
\def\j{{\mathbf{j}}}

\def\Racc{R^\circ}
\def\Rich_#1^#2{\Racc_{#1,#2}}
\def\Richcl_#1^#2{R_{#1,#2}}

\def\O{{\Ocal}}

\def\h{{\mathfrak{h}}}

\def\HH{H\!\!H}

\def\Fbb{{\mathbb{F}}}

\def\Pio{\Pi^\circ}
\def\Pit{\Xcal^\circ}

\def\BS{\operatorname{BS}}

\numberwithin{equation}{section}

\def\FLY{HOMFLY\xspace}

\def\Richafftp_#1^#2{{\Rcal_{#1,#2}^{>0}}}

\def\Richaff_#1^#2{\accentset{\circ}{\mathcal{R}}_{#1,#2}}

\def\Cat{C}

\def\Dyck{\operatorname{Dyck}}
\def\Dyckxx(#1,#2){\Dyck_{#1,#2}}

\def\ncyc{\operatorname{c}}

\def\Bkn{\BND_{k,n}}
\def\Bknc{\Theta_{k,n}}
\def\Bxc_#1{\Theta_{#1}}

\def\betah{\hat\beta}

\def\fkn{{f_{k,n}}}
\def\Pit_#1{\Xcal^\circ_{#1}}

\def\k{\mathbbm{k}}

\def\BS{\operatorname{BS}}

\def\FR{F^\bullet}

\def\top{{\operatorname{top}}}

\def\Povar_#1{\accentset{\circ}{\Pi}_{#1}}
\def\Povarcl_#1{\Pi_{#1}}
\def\RPovar_#1{\accentset{\circ}{\Pi}^\R_{#1}}
\def\RPovarcl_#1{\Pi^\R_{#1}}
\def\Povtp_#1{\Pi_{#1}^{>0}}
\def\Povtnn_#1{\Pi_{#1}^{\geq0}}

\def\Gr{\operatorname{Gr}}

\def\KLR_#1^#2{R_{#1,#2}(q)}

\def\top{{\operatorname{top}}}
\def\Ptop_#1{P^\top_{#1}(q)}
\def\Ptopnoq_#1{P^\top_{#1}}

\def\O{\Ocal}

\def\mda(#1){\deg^\top_a(#1)}

\def\BS_#1{B_{#1}}

\def\bul{\bullet}

\def\HHB#1#2(#3){H^{#1,(#2)}(\HH^0(#3))}
\def\HHBC#1#2(#3){H^{#1,(#2)}(\HHC^0(#3))}
\def\HHX#1#2_#3{H^{#1,(#2)}(\HH^0(\FR_{#3}))}
\def\HHXC#1#2_#3{H^{#1,(#2)}(\HHC^0(\FR_{#3}))}
\def\HHXk#1#2_#3{H^{#1,(#2)}(\HHk^0(\FR_{#3}))}

\def\ksa[#1]{[#1]}

\def\FR{F^\bul}

\def\eps{\epsilon}

\def\k{{\mathbbm{k}}}

\def\Rsemi_#1^#2{\Racc_{#1,\overline{#2}}}

\def\HHC{\HH_\C}
\def\HHk{\HH_\k}

\def\PGL{\operatorname{PGL}}

\def\Rgauge_#1^#2{R^{\circ,\Delta=1}_{#1,#2}}

\def\drawbox(#1,#2){
\draw[black!50,dashed] (#1-0.5,#2-0.5) rectangle (#1+0.5,#2+0.5);
}

\def\drawgrid#1#2#3#4{
\foreach\i in {#1,...,#3}{
\foreach \j in {#2,...,#4}{
\drawbox(\i,\j)
}
}
}

\def\crossing{\scalebox{0.4}{\begin{tikzpicture}[baseline=(ZUZU.base)]
\coordinate(ZUZU) at (0,-0.5);\drawgrid{0}{0}{0}{0}\draw[line width=3pt, rounded corners=12] (0.00,-0.50)--(0.00,0.50);\draw[line width=3pt, rounded corners=12] (-0.50,0.00)--(0.50,0.00);
\end{tikzpicture}}\xspace} 
\def\elbow{\scalebox{0.4}{\begin{tikzpicture}[baseline=(ZUZU.base),xscale=-1]
\coordinate(ZUZU) at (0,-0.5);\drawgrid{0}{0}{0}{0}\draw[line width=3pt, rounded corners=12] (0.00,-0.50)--(0.00,0.00)--(0.50,0.00);\draw[line width=3pt, rounded corners=12] (-0.50,0.00)--(0.00,0.00)--(0.00,0.50);
\end{tikzpicture}}\xspace}

\def\maxx{\operatorname{max}}
\def\Go{\operatorname{Deo}}
\def\Gom{\Go^{\maxx}}

\numberwithin{equation}{section}

\def\slope{\operatorname{slope}}
\def\ton{\to n}
\def\Pinf{P_{\infty}}
\def\Qinf{Q_{\infty}}

\def\Cat{C}

\makeatletter
\@namedef{subjclassname@2020}{\textup{2020} Mathematics Subject Classification}
\makeatother

\def\fb{\bar f}
\def\perm{\fb}
\def\Sn{S_n}
\def\Snc{{\rm Cyc}(n)}
\def\Sc_#1{{\rm Cyc}(#1)}

\def\fkn{{f_{k,n}}}
\def\Rt{\tilde R}

\def\k{k}
\def\n{n}
\def\kn{\del}
\def\fij{f^{(i,j)}}
\def\fbij{\fb^{(i,j)}}
\def\fii{f^{(i,i+1)}}

\def\DyckF(#1){\Dyck(\Fr(#1))}
\def\Pfinf{P^{(f)}_\infty}
\def\Pf{P^{(f)}}
\def\Pg{P^{(g)}}
\def\Pginf{P^{(g)}_\infty}

\def\pf_#1{p_{f,#1}}
\def\lab{j}

\def\qf_#1{q_{f,#1}}

\def\del{\delta}
\def\del{\delta}
\def\delp{\delta^\perp}

\let\slp\slope

\def\directref#1#2{\hyperref[#2]{\getrefnumber{#1}\itemreff{#2}}}

\def\PX#1{P^{(#1)}_\infty}
\def\cas#1{

 {\bf Case #1:}
}
\def\wsw{\preceq}
\def\sw{\prec}
\def\LX#1{L^{(#1)}}

\def\Del{\Delta}

\def\Hbf{H}
\def\H{H}
\def\h{h}
\def\lf{\lfloor}
\def\rf{\rfloor}
\def\Fbi{\tilde \Flet}

\def\gb{{\bar g}}
\def\Fmax{\Flet^{\maxx}_{k,n}}

\def\Fmin{\Flet^{\minn}_{k,n}}
\def\shf{\sigma f}
\def\difmod{\nu}
\def\difmodb{\bar\nu}
\def\epskn{\eps_{k,n}}
\def\a{{\mathbf{a}}}

\def\Cata{\Cat_{\a}}
\def\Cataf{\Cat_{\a(f)}}

\def\maxx{\operatorname{max}}
\def\Go{\operatorname{Deo}}
\def\Gom{\Go^{\maxx}}
\def\T{{\mathbb{T}}}
\def\Pb{\bar P}
\def\PHbf{P^{(\Hbf)}}
\def\PHbfinf{P^{(\Hbf)}_\infty}
\def\PHbfp{P^{(\Hbf')}}

\def\subfty{}
\def\rotn{$180^\circ$}

\def\Pb{\bar P}

\def\beps{\vec{\eps}}

\def\F{{\operatorname{\Flet}}'}
\def\Fbf{\tilde\Flet'(f)}
\def\Fb{\tilde\Flet'}
\def\Fr{{\operatorname{\Flet}}}

\def\tS{{\tilde S}}
\def\O{{\mathcal{O}}}
\def\Omin{\O_{{\rm min}}}

\def\knk{\gamma}
\def\Flet{\Gamma} %

\def\tSp{\tS}

\def\Ceqvt{c-equivalent\xspace}
\def\Ceqvce{c-equivalence\xspace}
\def\Ceqvces{c-equivalences\xspace}
\def\ceq{\stackrel{\scalebox{0.6}{\ \normalfont{c}}}{\sim}}
\def\capprox{\stackrel{\scalebox{0.6}{\ \normalfont{c}}}{\approx}}

\def\crossing{\scalebox{0.4}{\begin{tikzpicture}[baseline=(ZUZU.base)]
\coordinate(ZUZU) at (0,-0.5);\drawgrid{0}{0}{0}{0}\draw[line width=3pt, rounded corners=12] (0.00,-0.50)--(0.00,0.50);\draw[line width=3pt, rounded corners=12] (-0.50,0.00)--(0.50,0.00);
\end{tikzpicture}}\xspace} 
\def\elbow{\scalebox{0.4}{\begin{tikzpicture}[baseline=(ZUZU.base),xscale=-1]
\coordinate(ZUZU) at (0,-0.5);\drawgrid{0}{0}{0}{0}\draw[line width=3pt, rounded corners=12] (0.00,-0.50)--(0.00,0.00)--(0.50,0.00);\draw[line width=3pt, rounded corners=12] (-0.50,0.00)--(0.00,0.00)--(0.00,0.50);
\end{tikzpicture}}\xspace}

\def\repfree{repetition-free\xspace}
\def\Repfree{Repetition-free\xspace}

\def\tob{\to b}
\def\tod{\to d}
\def\Ra{R_\alpha}
\def\Rb{R_\beta}

\def\PZ{(\Del_1\cup\Del_2)_{\Z}}

\begin{document}
\numberwithin{equation}{section}
\title{Positroid Catalan numbers}
\author{Pavel Galashin}
\address{Department of Mathematics, University of California, Los Angeles, 520 Portola Plaza,
Los Angeles, CA 90025, USA}
\email{\href{mailto:galashin@math.ucla.edu}{galashin@math.ucla.edu}}

\author{Thomas Lam}
\address{Department of Mathematics, University of Michigan, 2074 East Hall, 530 Church Street, Ann Arbor, MI 48109-1043, USA}
\email{\href{mailto:tfylam@umich.edu}{tfylam@umich.edu}}
\thanks{P.G.\ was supported by an Alfred P. Sloan Research Fellowship and by the National Science Foundation under Grants No.~DMS-1954121 and No.~DMS-2046915. T.L.\ was supported by grants DMS-1464693 and DMS-1953852 from the National Science Foundation.}

\begin{abstract}
Given a permutation $f$, we study the \emph{positroid Catalan number} $C_f$ defined to be the torus-equivariant Euler characteristic of the associated open positroid variety. We introduce a class of \emph{\repfree permutations} and show that the corresponding positroid Catalan numbers count Dyck paths avoiding a convex subset of the rectangle. We show that any convex subset appears in this way.  Conjecturally, the associated $q,t$-polynomials coincide with the \emph{generalized $q,t$-Catalan numbers} that recently appeared in relation to the shuffle conjecture, flag Hilbert schemes, and Khovanov--Rozansky homology of Coxeter links.   %
\end{abstract}


\date{\today}

\maketitle

\setcounter{tocdepth}{1}

\vspace{-0.09in}

\tableofcontents

\maketitle

\section{Introduction}
\emph{Open positroid varieties} are remarkable subvarieties of the Grassmannian introduced by Knutson--Lam--Speyer in~\cite{KLS}, building on the work of Postnikov~\cite{Pos}.  They appear in numerous contexts: total positivity, Schubert calculus, Poisson geometry, scattering amplitudes, cluster algebras, and so on~\cite{LusIntro, BGY, abcgpt, positroidcluster}. 
In a recent paper~\cite{qtcat}, we further connected positroid varieties to knot invariants, showing a relation between the cohomology of an open positroid variety $\Pio_f$ and \emph{Khovanov--Rozansky homology}~\cite{KR1,KR2} of an associated \emph{positroid link} $\betah_f$.

\subsection{Positroid Catalan numbers}\label{sec:perm-affine-notat}
Let $\Snc$ denote the set of $n$-cycles in the symmetric group $S_n$.  To each $\perm\in\Snc$ we associate a \emph{bounded affine permutation} $f:\Z\to\Z$. The map $f$ is uniquely determined by the conditions $f(i+n)=f(i)+n$ and $i<f(i)<i+n$ for all $i\in\Z$, together with $f(i)\equiv\perm(i)\pmod n$ for all $1\leq i\leq n$. Taking $f$ modulo $n$ recovers $\perm$, and thus $f$ and $\perm$ determine each other.  See \cref{fig:aff_notn} for an example and \cref{sec:affine-permutations} for further details.

For a bounded affine permutation $f $, let $\Pio_f \subset \Gr(k,n)$ denote the corresponding \emph{open positroid variety} of the Grassmannian.  Let $T \subset \PGL(n)$ denote the natural torus of diagonal matrices acting on $\Gr(k,n)$.

\begin{definition}\label{defn:intro}
For an $n$-cycle $\perm \in \Snc$, define the \emph{positroid Catalan number} 
$$
\Cat_f := \chi_{T}(\Pio_f)
$$ 
to be the torus-equivariant Euler characteristic of $\Pio_f$. 
\end{definition}

 These numbers are positive integers which can be computed via an explicit combinatorial recurrence; see \cref{sec:recurs-positr-catal}. Additionally, they have the following interpretations:
\begin{enumerate}[label=(\alph*)]
\item $\Cat_f$ is equal to the number of \emph{maximal $f$-Deograms} introduced in~\cite{qtcat}, which are in bijection with a class of distinguished subexpressions in the sense of Deodhar \cite{Deodhar}; see \cref{sec:deograms}.
\item $\Cat_f$ is equal to the $q = 1$ evaluation of the polynomial $\Rt_f(q):=R_f(q)/(q-1)^{n-1}$, where $R_f(q)$ is the \emph{Kazhdan--Lusztig $R$-polynomial}~\cite{KL1,KL2}; see \cref{sec:affine_R_poly}. By \cite[Theorem~1.11]{qtcat}, $\Rt_f(q)$ may be obtained as a coefficient of the \emph{HOMFLY polynomial}~\cite{HOMFLY,PT} of $\betah_f$.
\item  $\Cat_f$ is equal to the $q=t=1$ evaluation of the \emph{mixed Hodge polynomial} $\Pcal(\Pio_f/T;q,t)$.  By \cite{qtcat}, ${\mathcal P}(\Pio_f/T;q,t)$ is equal to a coefficient of the Khovanov--Rozansky triply-graded link invariant of $\betah_f$.
\end{enumerate}

We showed in~\cite{qtcat}, using results on torus knots that date back to Jones \cite{Jones}, that for $\gcd(k,n)=1$ and $f$ given by $f(i) = i+k$,  %
the positroid Catalan number $\Cat_f$ recovers the famous (rational) Catalan number $\Cat_{k,n-k} := \frac{1}{n}\binom{n}{k}$ which counts Dyck paths above the diagonal inside a $k\times (n-k)$ rectangle.  This explains the nomenclature ``positroid Catalan number" and points towards a deeper investigation of positroid Catalan numbers from a combinatorial perspective.  
In this work, we make the first step in this direction.  

\begin{figure}

\setlength{\tabcolsep}{0pt}
\begin{tabular}{ccc}
\scalebox{0.60}{
\begin{tikzpicture}[xscale=0.50, yscale=0.50, baseline=(ZUZU.base)]
	\coordinate(ZUZU) at (0,1.50);
\node[draw,circle,scale=0.30,fill=black] (T0) at (0,3.00) {};
\node[anchor=south,scale=1.00] (LT0) at (0,3.00) {$0$};
\node[draw,circle,scale=0.30,fill=black] (T1) at (1,3.00) {};
\node[anchor=south,scale=1.00] (LT1) at (1,3.00) {$1$};
\node[draw,circle,scale=0.30,fill=black] (T2) at (2,3.00) {};
\node[anchor=south,scale=1.00] (LT2) at (2,3.00) {$2$};
\node[draw,circle,scale=0.30,fill=black] (T3) at (3,3.00) {};
\node[anchor=south,scale=1.00] (LT3) at (3,3.00) {$3$};
\node[draw,circle,scale=0.30,fill=black] (T4) at (4,3.00) {};
\node[anchor=south,scale=1.00] (LT4) at (4,3.00) {$4$};
\node[draw,circle,scale=0.30,fill=black] (T5) at (5,3.00) {};
\node[anchor=south,scale=1.00] (LT5) at (5,3.00) {$5$};
\node[draw,circle,scale=0.30,fill=black] (T6) at (6,3.00) {};
\node[anchor=south,scale=1.00] (LT6) at (6,3.00) {$6$};
\node[draw,circle,scale=0.30,fill=black] (B0) at (0,0.00) {};
\node[anchor=north,scale=1.00] (LT0) at (0,0.00) {$0$};
\node[draw,circle,scale=0.30,fill=black] (B1) at (1,0.00) {};
\node[anchor=north,scale=1.00] (LT1) at (1,0.00) {$1$};
\node[draw,circle,scale=0.30,fill=black] (B2) at (2,0.00) {};
\node[anchor=north,scale=1.00] (LT2) at (2,0.00) {$2$};
\node[draw,circle,scale=0.30,fill=black] (B3) at (3,0.00) {};
\node[anchor=north,scale=1.00] (LT3) at (3,0.00) {$3$};
\node[draw,circle,scale=0.30,fill=black] (B4) at (4,0.00) {};
\node[anchor=north,scale=1.00] (LT4) at (4,0.00) {$4$};
\node[draw,circle,scale=0.30,fill=black] (B5) at (5,0.00) {};
\node[anchor=north,scale=1.00] (LT5) at (5,0.00) {$5$};
\node[draw,circle,scale=0.30,fill=black] (B6) at (6,0.00) {};
\node[anchor=north,scale=1.00] (LT6) at (6,0.00) {$6$};
\draw[black, line width=0.50pt] (T0)--(B2);
\draw[black, line width=0.50pt] (T1)--(B0);
\draw[black, line width=0.50pt] (T2)--(B3);
\draw[black, line width=0.50pt] (T3)--(B6);
\draw[black, line width=0.50pt] (T4)--(B5);
\draw[black, line width=0.50pt] (T5)--(B1);
\draw[black, line width=0.50pt] (T6)--(B4);
\end{tikzpicture}

}
& \quad $\longrightarrow$ \quad &
\scalebox{0.60}{
\begin{tikzpicture}[xscale=0.50, yscale=0.50, baseline=(ZUZU.base)]
	\coordinate(ZUZU) at (0,1.50);
\node[draw,circle,scale=0.30,fill=black] (T0) at (0,3.00) {};
\node[anchor=south,scale=1.00] (LT0) at (0,3.00) {$0$};
\node[draw,circle,scale=0.30,fill=black] (T1) at (1,3.00) {};
\node[anchor=south,scale=1.00] (LT1) at (1,3.00) {$1$};
\node[draw,circle,scale=0.30,fill=black] (T2) at (2,3.00) {};
\node[anchor=south,scale=1.00] (LT2) at (2,3.00) {$2$};
\node[draw,circle,scale=0.30,fill=black] (T3) at (3,3.00) {};
\node[anchor=south,scale=1.00] (LT3) at (3,3.00) {$3$};
\node[draw,circle,scale=0.30,fill=black] (T4) at (4,3.00) {};
\node[anchor=south,scale=1.00] (LT4) at (4,3.00) {$4$};
\node[draw,circle,scale=0.30,fill=black] (T5) at (5,3.00) {};
\node[anchor=south,scale=1.00] (LT5) at (5,3.00) {$5$};
\node[draw,circle,scale=0.30,fill=black] (T6) at (6,3.00) {};
\node[anchor=south,scale=1.00] (LT6) at (6,3.00) {$6$};
\node[draw,circle,scale=0.30,fill=black] (T7) at (7,3.00) {};
\node[anchor=south,scale=1.00] (LT7) at (7,3.00) {$0$};
\node[draw,circle,scale=0.30,fill=black] (T8) at (8,3.00) {};
\node[anchor=south,scale=1.00] (LT8) at (8,3.00) {$1$};
\node[draw,circle,scale=0.30,fill=black] (T9) at (9,3.00) {};
\node[anchor=south,scale=1.00] (LT9) at (9,3.00) {$2$};
\node[draw,circle,scale=0.30,fill=black] (T10) at (10,3.00) {};
\node[anchor=south,scale=1.00] (LT10) at (10,3.00) {$3$};
\node[draw,circle,scale=0.30,fill=black] (T11) at (11,3.00) {};
\node[anchor=south,scale=1.00] (LT11) at (11,3.00) {$4$};
\node[draw,circle,scale=0.30,fill=black] (T12) at (12,3.00) {};
\node[anchor=south,scale=1.00] (LT12) at (12,3.00) {$5$};
\node[draw,circle,scale=0.30,fill=black] (T13) at (13,3.00) {};
\node[anchor=south,scale=1.00] (LT13) at (13,3.00) {$6$};
\node[draw,circle,scale=0.30,fill=black] (T14) at (14,3.00) {};
\node[anchor=south,scale=1.00] (LT14) at (14,3.00) {$0$};
\node[draw,circle,scale=0.30,fill=black] (T15) at (15,3.00) {};
\node[anchor=south,scale=1.00] (LT15) at (15,3.00) {$1$};
\node[draw,circle,scale=0.30,fill=black] (T16) at (16,3.00) {};
\node[anchor=south,scale=1.00] (LT16) at (16,3.00) {$2$};
\node[draw,circle,scale=0.30,fill=black] (T17) at (17,3.00) {};
\node[anchor=south,scale=1.00] (LT17) at (17,3.00) {$3$};
\node[draw,circle,scale=0.30,fill=black] (T18) at (18,3.00) {};
\node[anchor=south,scale=1.00] (LT18) at (18,3.00) {$4$};
\node[draw,circle,scale=0.30,fill=black] (B2) at (2,0.00) {};
\node[anchor=north,scale=1.00] (LT2) at (2,0.00) {$2$};
\node[draw,circle,scale=0.30,fill=black] (B3) at (3,0.00) {};
\node[anchor=north,scale=1.00] (LT3) at (3,0.00) {$3$};
\node[draw,circle,scale=0.30,fill=black] (B4) at (4,0.00) {};
\node[anchor=north,scale=1.00] (LT4) at (4,0.00) {$4$};
\node[draw,circle,scale=0.30,fill=black] (B5) at (5,0.00) {};
\node[anchor=north,scale=1.00] (LT5) at (5,0.00) {$5$};
\node[draw,circle,scale=0.30,fill=black] (B6) at (6,0.00) {};
\node[anchor=north,scale=1.00] (LT6) at (6,0.00) {$6$};
\node[draw,circle,scale=0.30,fill=black] (B7) at (7,0.00) {};
\node[anchor=north,scale=1.00] (LT7) at (7,0.00) {$0$};
\node[draw,circle,scale=0.30,fill=black] (B8) at (8,0.00) {};
\node[anchor=north,scale=1.00] (LT8) at (8,0.00) {$1$};
\node[draw,circle,scale=0.30,fill=black] (B9) at (9,0.00) {};
\node[anchor=north,scale=1.00] (LT9) at (9,0.00) {$2$};
\node[draw,circle,scale=0.30,fill=black] (B10) at (10,0.00) {};
\node[anchor=north,scale=1.00] (LT10) at (10,0.00) {$3$};
\node[draw,circle,scale=0.30,fill=black] (B11) at (11,0.00) {};
\node[anchor=north,scale=1.00] (LT11) at (11,0.00) {$4$};
\node[draw,circle,scale=0.30,fill=black] (B12) at (12,0.00) {};
\node[anchor=north,scale=1.00] (LT12) at (12,0.00) {$5$};
\node[draw,circle,scale=0.30,fill=black] (B13) at (13,0.00) {};
\node[anchor=north,scale=1.00] (LT13) at (13,0.00) {$6$};
\node[draw,circle,scale=0.30,fill=black] (B14) at (14,0.00) {};
\node[anchor=north,scale=1.00] (LT14) at (14,0.00) {$0$};
\node[draw,circle,scale=0.30,fill=black] (B15) at (15,0.00) {};
\node[anchor=north,scale=1.00] (LT15) at (15,0.00) {$1$};
\node[draw,circle,scale=0.30,fill=black] (B16) at (16,0.00) {};
\node[anchor=north,scale=1.00] (LT16) at (16,0.00) {$2$};
\node[draw,circle,scale=0.30,fill=black] (B17) at (17,0.00) {};
\node[anchor=north,scale=1.00] (LT17) at (17,0.00) {$3$};
\node[draw,circle,scale=0.30,fill=black] (B18) at (18,0.00) {};
\node[anchor=north,scale=1.00] (LT18) at (18,0.00) {$4$};
\node[draw,circle,scale=0.30,fill=black] (B19) at (19,0.00) {};
\node[anchor=north,scale=1.00] (LT19) at (19,0.00) {$5$};
\node[draw,circle,scale=0.30,fill=black] (B20) at (20,0.00) {};
\node[anchor=north,scale=1.00] (LT20) at (20,0.00) {$6$};
\draw[black, line width=0.50pt] (T0)--(B2);
\draw[black, line width=0.50pt] (T1)--(B7);
\draw[black, line width=0.50pt] (T2)--(B3);
\draw[black, line width=0.50pt] (T3)--(B6);
\draw[black, line width=0.50pt] (T4)--(B5);
\draw[black, line width=0.50pt] (T5)--(B8);
\draw[black, line width=0.50pt] (T6)--(B11);
\draw[black, line width=0.50pt] (T7)--(B9);
\draw[black, line width=0.50pt] (T8)--(B14);
\draw[black, line width=0.50pt] (T9)--(B10);
\draw[black, line width=0.50pt] (T10)--(B13);
\draw[black, line width=0.50pt] (T11)--(B12);
\draw[black, line width=0.50pt] (T12)--(B15);
\draw[black, line width=0.50pt] (T13)--(B18);
\draw[black, line width=0.50pt] (T14)--(B16);
\draw[black, line width=0.50pt] (T16)--(B17);
\draw[black, line width=0.50pt] (T17)--(B20);
\draw[black, line width=0.50pt] (T18)--(B19);
\node[scale=1.50, anchor=west] (DDDTR) at (18.30, 3.00) {$\dots$};
\node[scale=1.50, anchor=east] (DDDTR) at (-0.30, 3.00) {$\dots$};
\node[scale=1.50, anchor=west] (DDDTR) at (20.30, 0.00) {$\dots$};
\node[scale=1.50, anchor=east] (DDDTR) at (1.70, 0.00) {$\dots$};
\end{tikzpicture}

}
\\ \vspace{-0.1in} \\
permutation $\fb\in\Snc$ & & bounded affine permutation $f\in\Bknc$
\end{tabular}
  \caption{\label{fig:aff_notn} Drawing an $n$-cycle (left) in affine notation (right).}
\end{figure}
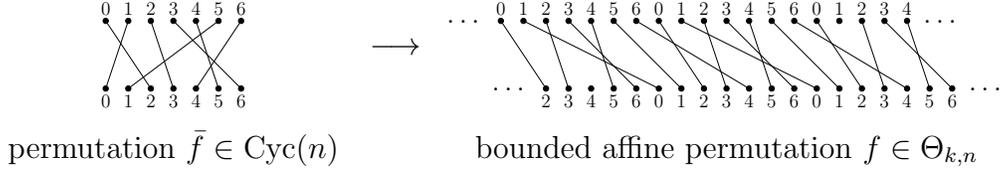

\subsection{\Repfree permutations}\label{sec:intro:rep-free}  To each $n$-cycle $\perm\in\Snc$ we associate an \emph{inversion multiset} $\Fr(f)$, and we introduce a natural class of \emph{\repfree} permutations for which the multiset $\Fr(f)$ has no repeated elements.
Let $[n]:=\{1,2,\dots,n\}$ and for $\perm\in\Snc$, set %
\begin{equation}\label{eq:intro:wind_dfn}
  \k(\perm):=\#\{i\in[n]\mid \perm(i)<i\}.
\end{equation}
For $1\leq k\leq n-1$, we denote
\begin{equation*}%
  \Bknc:=\{f\mid \perm\in \Snc\text{ is such that $\k(\perm)=k$}\}.
\end{equation*}
The set $\Snc$ is in bijection with $\bigsqcup_{k=1}^{n-1}\Bknc$. %

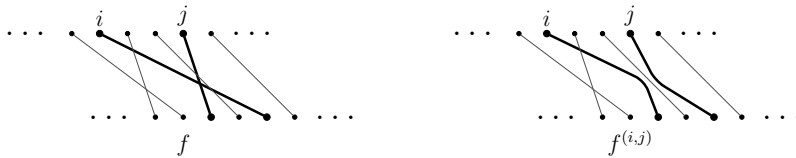
\begin{figure}

  \resizebox{0.7\textwidth}{!}{
\def\rc{5}
\begin{tabular}{ccc}
\begin{tikzpicture}[xscale=0.50, yscale=0.50, baseline=(ZUZU.base)]
	\coordinate(ZUZU) at (0,1.50);
\node[draw,circle,scale=0.20,fill=black] (T0) at (0,3.00) {};
\node[draw,circle,scale=0.3,fill=black] (T1) at (1,3.00) {};
\node[anchor=south,scale=1.00] (LT1) at (1,3.00) {$i$};
\node[draw,circle,scale=0.20,fill=black] (T2) at (2,3.00) {};
\node[draw,circle,scale=0.20,fill=black] (T3) at (3,3.00) {};
\node[draw,circle,scale=0.3,fill=black] (T4) at (4,3.00) {};
\node[anchor=south,scale=1.00] (LT4) at (4,3.00) {$j$};
\node[draw,circle,scale=0.20,fill=black] (T5) at (5,3.00) {};
\node[draw,circle,scale=0.20,fill=black] (B3) at (3,0.00) {};
\node[draw,circle,scale=0.20,fill=black] (B4) at (4,0.00) {};
\node[draw,circle,scale=0.3,fill=black] (B5) at (5,0.00) {};
\node[draw,circle,scale=0.20,fill=black] (B6) at (6,0.00) {};
\node[draw,circle,scale=0.3,fill=black] (B7) at (7,0.00) {};
\node[draw,circle,scale=0.20,fill=black] (B8) at (8,0.00) {};
\draw[line width=0.3pt, black!70] (T0)--(B4);
\draw[line width=1.3pt] (T1)--(B7);
\draw[line width=1.3pt] (T4)--(B5);
\draw[line width=0.3pt, black!70] (T2)--(B3);
\draw[line width=0.3pt, black!70] (T3)--(B6);
\draw[line width=0.3pt, black!70] (T5)--(B8);
\node[scale=1.50, anchor=west] (DDDTR) at (5.30, 3.00) {$\dots$};
\node[scale=1.50, anchor=east] (DDDTR) at (-0.30, 3.00) {$\dots$};
\node[scale=1.50, anchor=west] (DDDTR) at (8.30, 0.00) {$\dots$};
\node[scale=1.50, anchor=east] (DDDTR) at (2.70, 0.00) {$\dots$};
\end{tikzpicture}

&\hspace{0.2in} &
\begin{tikzpicture}[xscale=0.50, yscale=0.50, baseline=(ZUZU.base)]
	\coordinate(ZUZU) at (0,1.50);
\node[draw,circle,scale=0.20,fill=black] (T0) at (0,3.00) {};
\node[draw,circle,scale=0.3,fill=black] (T1) at (1,3.00) {};
\node[anchor=south,scale=1.00] (LT1) at (1,3.00) {$i$};
\node[draw,circle,scale=0.20,fill=black] (T2) at (2,3.00) {};
\node[draw,circle,scale=0.20,fill=black] (T3) at (3,3.00) {};
\node[draw,circle,scale=0.3,fill=black] (T4) at (4,3.00) {};
\node[anchor=south,scale=1.00] (LT4) at (4,3.00) {$j$};
\node[draw,circle,scale=0.20,fill=black] (T5) at (5,3.00) {};
\node[draw,circle,scale=0.20,fill=black] (B3) at (3,0.00) {};
\node[draw,circle,scale=0.20,fill=black] (B4) at (4,0.00) {};
\node[draw,circle,scale=0.3,fill=black] (B5) at (5,0.00) {};
\node[draw,circle,scale=0.20,fill=black] (B6) at (6,0.00) {};
\node[draw,circle,scale=0.3,fill=black] (B7) at (7,0.00) {};
\node[draw,circle,scale=0.20,fill=black] (B8) at (8,0.00) {};
\draw[line width=0.3pt, black!70] (T0)--(B4);
\draw[line width=1.3pt,rounded corners=\rc] (T1)--(4.5,1.3)--(B5);
\draw[line width=1.3pt,rounded corners=\rc] (T4)--(4.9,1.3)--(B7);
\draw[line width=0.3pt, black!70] (T2)--(B3);
\draw[line width=0.3pt, black!70] (T3)--(B6);
\draw[line width=0.3pt, black!70] (T5)--(B8);
\node[scale=1.50, anchor=west] (DDDTR) at (5.30, 3.00) {$\dots$};
\node[scale=1.50, anchor=east] (DDDTR) at (-0.30, 3.00) {$\dots$};
\node[scale=1.50, anchor=west] (DDDTR) at (8.30, 0.00) {$\dots$};
\node[scale=1.50, anchor=east] (DDDTR) at (2.70, 0.00) {$\dots$};
\end{tikzpicture}
 \\
$f$ && $\fij$
\end{tabular}

}
  \caption{\label{fig:resolve} Resolving a crossing $(i,j)$ of $f$.}
\end{figure}
For $f\in\Bknc$, let %
\begin{equation*}%
  \k(f)=\k(\fb):=k,\quad \n(f)=\n(\fb):=n, \quad\text{and}\quad \knk(f)=\knk(\fb):=(k,n-k).
\end{equation*}
An \emph{inversion} of $f\in\Bknc$ is a pair $(i,j)$ of integers such that $i<j$, $f(i)>f(j)$, and $i\in[n]$. 
The \emph{length} $\ell(f)$ is the number of inversions of $f$. 
For an inversion $(i,j)$ of $f$, let $\fij:\Z\to\Z$ be obtained by swapping the values $f(i)$ and $f(j)$ (and repeating this for $f(i+rn)$ and $f(j+rn)$ for all $r\in\Z$).  We say that $\fij$ is obtained from $f$ by \emph{resolving the crossing $(i,j)$}; see \cref{fig:resolve}. We let $\fbij\in S_n$ denote the permutation obtained by reducing $\fij$ modulo~$n$.

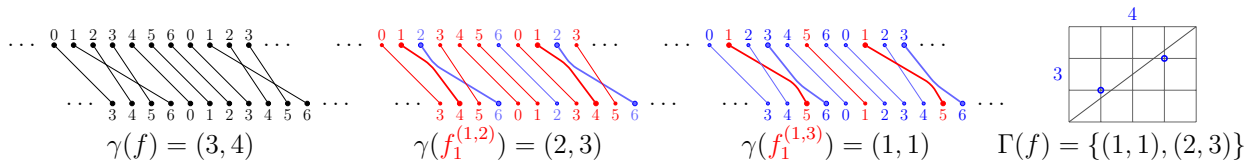
\begin{figure}

\resizebox{\textwidth}{!}{
\setlength{\tabcolsep}{-10pt}
\def\rc{5}
\begin{tabular}{ccccc}
\scalebox{0.60}{
\begin{tikzpicture}[xscale=0.50, yscale=0.50, baseline=(ZUZU.base)]
	\coordinate(ZUZU) at (0,1.50);
\node[draw,circle,scale=0.30,fill=black] (T0) at (0,3.00) {};
\node[anchor=south,scale=1.00] (LT0) at (0,3.00) {$0$};
\node[draw,circle,scale=0.30,fill=black] (T1) at (1,3.00) {};
\node[anchor=south,scale=1.00] (LT1) at (1,3.00) {$1$};
\node[draw,circle,scale=0.30,fill=black] (T2) at (2,3.00) {};
\node[anchor=south,scale=1.00] (LT2) at (2,3.00) {$2$};
\node[draw,circle,scale=0.30,fill=black] (T3) at (3,3.00) {};
\node[anchor=south,scale=1.00] (LT3) at (3,3.00) {$3$};
\node[draw,circle,scale=0.30,fill=black] (T4) at (4,3.00) {};
\node[anchor=south,scale=1.00] (LT4) at (4,3.00) {$4$};
\node[draw,circle,scale=0.30,fill=black] (T5) at (5,3.00) {};
\node[anchor=south,scale=1.00] (LT5) at (5,3.00) {$5$};
\node[draw,circle,scale=0.30,fill=black] (T6) at (6,3.00) {};
\node[anchor=south,scale=1.00] (LT6) at (6,3.00) {$6$};
\node[draw,circle,scale=0.30,fill=black] (T7) at (7,3.00) {};
\node[anchor=south,scale=1.00] (LT7) at (7,3.00) {$0$};
\node[draw,circle,scale=0.30,fill=black] (T8) at (8,3.00) {};
\node[anchor=south,scale=1.00] (LT8) at (8,3.00) {$1$};
\node[draw,circle,scale=0.30,fill=black] (T9) at (9,3.00) {};
\node[anchor=south,scale=1.00] (LT9) at (9,3.00) {$2$};
\node[draw,circle,scale=0.30,fill=black] (T10) at (10,3.00) {};
\node[anchor=south,scale=1.00] (LT10) at (10,3.00) {$3$};
\node[draw,circle,scale=0.30,fill=black] (B3) at (3,0.00) {};
\node[anchor=north,scale=1.00] (LT3) at (3,0.00) {$3$};
\node[draw,circle,scale=0.30,fill=black] (B4) at (4,0.00) {};
\node[anchor=north,scale=1.00] (LT4) at (4,0.00) {$4$};
\node[draw,circle,scale=0.30,fill=black] (B5) at (5,0.00) {};
\node[anchor=north,scale=1.00] (LT5) at (5,0.00) {$5$};
\node[draw,circle,scale=0.30,fill=black] (B6) at (6,0.00) {};
\node[anchor=north,scale=1.00] (LT6) at (6,0.00) {$6$};
\node[draw,circle,scale=0.30,fill=black] (B7) at (7,0.00) {};
\node[anchor=north,scale=1.00] (LT7) at (7,0.00) {$0$};
\node[draw,circle,scale=0.30,fill=black] (B8) at (8,0.00) {};
\node[anchor=north,scale=1.00] (LT8) at (8,0.00) {$1$};
\node[draw,circle,scale=0.30,fill=black] (B9) at (9,0.00) {};
\node[anchor=north,scale=1.00] (LT9) at (9,0.00) {$2$};
\node[draw,circle,scale=0.30,fill=black] (B10) at (10,0.00) {};
\node[anchor=north,scale=1.00] (LT10) at (10,0.00) {$3$};
\node[draw,circle,scale=0.30,fill=black] (B11) at (11,0.00) {};
\node[anchor=north,scale=1.00] (LT11) at (11,0.00) {$4$};
\node[draw,circle,scale=0.30,fill=black] (B12) at (12,0.00) {};
\node[anchor=north,scale=1.00] (LT12) at (12,0.00) {$5$};
\node[draw,circle,scale=0.30,fill=black] (B13) at (13,0.00) {};
\node[anchor=north,scale=1.00] (LT13) at (13,0.00) {$6$};
\draw[black, line width=0.50pt] (T0)--(B3);
\draw[black, line width=0.50pt] (T1)--(B6);
\draw[black, line width=0.50pt] (T2)--(B4);
\draw[black, line width=0.50pt] (T3)--(B5);
\draw[black, line width=0.50pt] (T4)--(B7);
\draw[black, line width=0.50pt] (T5)--(B8);
\draw[black, line width=0.50pt] (T6)--(B9);
\draw[black, line width=0.50pt] (T7)--(B10);
\draw[black, line width=0.50pt] (T8)--(B13);
\draw[black, line width=0.50pt] (T9)--(B11);
\draw[black, line width=0.50pt] (T10)--(B12);
\node[scale=1.50, anchor=west] (DDDTR) at (10.30, 3.00) {$\dots$};
\node[scale=1.50, anchor=east] (DDDTR) at (-0.30, 3.00) {$\dots$};
\node[scale=1.50, anchor=west] (DDDTR) at (13.30, 0.00) {$\dots$};
\node[scale=1.50, anchor=east] (DDDTR) at (2.70, 0.00) {$\dots$};
\end{tikzpicture}

}
&
\scalebox{0.60}{
\begin{tikzpicture}[xscale=0.50, yscale=0.50, baseline=(ZUZU.base)]
	\coordinate(ZUZU) at (0,1.50);
\node[draw=red,circle,scale=0.2, fill=red] (T0) at (0,3.00) {};
\node[anchor=south,red,scale=1.00] (LT0) at (0,3.00) {$0$};
\node[draw=red,circle,scale=0.3,fill=red] (T1) at (1,3.00) {};
\node[anchor=south,red,scale=1.00] (LT1) at (1,3.00) {$1$};
\node[draw=blue,circle,scale=0.3,fill=blue!60] (T2) at (2,3.00) {};
\node[anchor=south,blue!60,scale=1.00] (LT2) at (2,3.00) {$2$};
\node[draw=red,circle,scale=0.2, fill=red] (T3) at (3,3.00) {};
\node[anchor=south,red,scale=1.00] (LT3) at (3,3.00) {$3$};
\node[draw=red,circle,scale=0.2, fill=red] (T4) at (4,3.00) {};
\node[anchor=south,red,scale=1.00] (LT4) at (4,3.00) {$4$};
\node[draw=red,circle,scale=0.2, fill=red] (T5) at (5,3.00) {};
\node[anchor=south,red,scale=1.00] (LT5) at (5,3.00) {$5$};
\node[draw=blue,circle,scale=0.2, fill=blue!60] (T6) at (6,3.00) {};
\node[anchor=south,blue!60,scale=1.00] (LT6) at (6,3.00) {$6$};
\node[draw=red,circle,scale=0.2, fill=red] (T7) at (7,3.00) {};
\node[anchor=south,red,scale=1.00] (LT7) at (7,3.00) {$0$};
\node[draw=red,circle,scale=0.3,fill=red] (T8) at (8,3.00) {};
\node[anchor=south,red,scale=1.00] (LT8) at (8,3.00) {$1$};
\node[draw=blue,circle,scale=0.3,fill=blue!60] (T9) at (9,3.00) {};
\node[anchor=south,blue!60,scale=1.00] (LT9) at (9,3.00) {$2$};
\node[draw=red,circle,scale=0.2, fill=red] (T10) at (10,3.00) {};
\node[anchor=south,red,scale=1.00] (LT10) at (10,3.00) {$3$};
\node[draw=red,circle,scale=0.2, fill=red] (B3) at (3,0.00) {};
\node[anchor=north,red,scale=1.00] (LT3) at (3,0.00) {$3$};
\node[draw=red,circle,scale=0.3,fill=red] (B4) at (4,0.00) {};
\node[anchor=north,red,scale=1.00] (LT4) at (4,0.00) {$4$};
\node[draw=red,circle,scale=0.2, fill=red] (B5) at (5,0.00) {};
\node[anchor=north,red,scale=1.00] (LT5) at (5,0.00) {$5$};
\node[draw=blue,circle,scale=0.3,fill=blue!60] (B6) at (6,0.00) {};
\node[anchor=north,blue!60,scale=1.00] (LT6) at (6,0.00) {$6$};
\node[draw=red,circle,scale=0.2, fill=red] (B7) at (7,0.00) {};
\node[anchor=north,red,scale=1.00] (LT7) at (7,0.00) {$0$};
\node[draw=red,circle,scale=0.2, fill=red] (B8) at (8,0.00) {};
\node[anchor=north,red,scale=1.00] (LT8) at (8,0.00) {$1$};
\node[draw=blue,circle,scale=0.2, fill=blue!60] (B9) at (9,0.00) {};
\node[anchor=north,blue!60,scale=1.00] (LT9) at (9,0.00) {$2$};
\node[draw=red,circle,scale=0.2, fill=red] (B10) at (10,0.00) {};
\node[anchor=north,red,scale=1.00] (LT10) at (10,0.00) {$3$};
\node[draw=red,circle,scale=0.3,fill=red] (B11) at (11,0.00) {};
\node[anchor=north,red,scale=1.00] (LT11) at (11,0.00) {$4$};
\node[draw=red,circle,scale=0.2, fill=red] (B12) at (12,0.00) {};
\node[anchor=north,red,scale=1.00] (LT12) at (12,0.00) {$5$};
\node[draw=blue,circle,scale=0.3,fill=blue!60] (B13) at (13,0.00) {};
\node[anchor=north,blue!60,scale=1.00] (LT13) at (13,0.00) {$6$};
\draw[line width=0.5pt, red] (T0)--(B3);
\draw[line width=1.3pt,red,rounded corners=\rc] (T1)--(2.5,2)--(B4);
\draw[line width=1.3pt,blue!60,rounded corners=\rc] (T2)--(2.7,2)--(B6);
\draw[line width=0.5pt, red] (T3)--(B5);
\draw[line width=0.5pt, red] (T4)--(B7);
\draw[line width=0.5pt, red] (T5)--(B8);
\draw[line width=0.5pt, blue!60] (T6)--(B9);
\draw[line width=0.5pt, red] (T7)--(B10);
\draw[line width=1.3pt,red,rounded corners=\rc] (T8)--(9.5,2)--(B11);
\draw[line width=1.3pt,blue!60,rounded corners=\rc] (T9)--(9.7,2)--(B13);
\draw[line width=0.5pt, red] (T10)--(B12);
\node[scale=1.50, anchor=west] (DDDTR) at (10.30, 3.00) {$\dots$};
\node[scale=1.50, anchor=east] (DDDTR) at (-0.30, 3.00) {$\dots$};
\node[scale=1.50, anchor=west] (DDDTR) at (13.30, 0.00) {$\dots$};
\node[scale=1.50, anchor=east] (DDDTR) at (2.70, 0.00) {$\dots$};
\end{tikzpicture}

}
&
\scalebox{0.60}{
\begin{tikzpicture}[xscale=0.50, yscale=0.50, baseline=(ZUZU.base)]
	\coordinate(ZUZU) at (0,1.50);
\node[draw=blue,circle,scale=0.2, fill=blue!60] (T0) at (0,3.00) {};
\node[anchor=south,blue,scale=1.00] (LT0) at (0,3.00) {$0$};
\node[draw=red,circle,scale=0.3,fill=red] (T1) at (1,3.00) {};
\node[anchor=south,red,scale=1.00] (LT1) at (1,3.00) {$1$};
\node[draw=blue,circle,scale=0.2, fill=blue!60] (T2) at (2,3.00) {};
\node[anchor=south,blue,scale=1.00] (LT2) at (2,3.00) {$2$};
\node[draw=blue,circle,scale=0.3,fill=blue!60] (T3) at (3,3.00) {};
\node[anchor=south,blue,scale=1.00] (LT3) at (3,3.00) {$3$};
\node[draw=blue,circle,scale=0.2, fill=blue!60] (T4) at (4,3.00) {};
\node[anchor=south,blue,scale=1.00] (LT4) at (4,3.00) {$4$};
\node[draw=red,circle,scale=0.2, fill=red] (T5) at (5,3.00) {};
\node[anchor=south,red,scale=1.00] (LT5) at (5,3.00) {$5$};
\node[draw=blue,circle,scale=0.2, fill=blue!60] (T6) at (6,3.00) {};
\node[anchor=south,blue,scale=1.00] (LT6) at (6,3.00) {$6$};
\node[draw=blue,circle,scale=0.2, fill=blue!60] (T7) at (7,3.00) {};
\node[anchor=south,blue,scale=1.00] (LT7) at (7,3.00) {$0$};
\node[draw=red,circle,scale=0.3,fill=red] (T8) at (8,3.00) {};
\node[anchor=south,red,scale=1.00] (LT8) at (8,3.00) {$1$};
\node[draw=blue,circle,scale=0.2, fill=blue!60] (T9) at (9,3.00) {};
\node[anchor=south,blue,scale=1.00] (LT9) at (9,3.00) {$2$};
\node[draw=blue,circle,scale=0.3,fill=blue!60] (T10) at (10,3.00) {};
\node[anchor=south,blue,scale=1.00] (LT10) at (10,3.00) {$3$};
\node[draw=blue,circle,scale=0.2, fill=blue!60] (B3) at (3,0.00) {};
\node[anchor=north,blue,scale=1.00] (LT3) at (3,0.00) {$3$};
\node[draw=blue,circle,scale=0.2, fill=blue!60] (B4) at (4,0.00) {};
\node[anchor=north,blue,scale=1.00] (LT4) at (4,0.00) {$4$};
\node[draw=red,circle,scale=0.3,fill=red] (B5) at (5,0.00) {};
\node[anchor=north,red,scale=1.00] (LT5) at (5,0.00) {$5$};
\node[draw=blue,circle,scale=0.3,fill=blue!60] (B6) at (6,0.00) {};
\node[anchor=north,blue,scale=1.00] (LT6) at (6,0.00) {$6$};
\node[draw=blue,circle,scale=0.2, fill=blue!60] (B7) at (7,0.00) {};
\node[anchor=north,blue,scale=1.00] (LT7) at (7,0.00) {$0$};
\node[draw=red,circle,scale=0.2, fill=red] (B8) at (8,0.00) {};
\node[anchor=north,red,scale=1.00] (LT8) at (8,0.00) {$1$};
\node[draw=blue,circle,scale=0.2, fill=blue!60] (B9) at (9,0.00) {};
\node[anchor=north,blue,scale=1.00] (LT9) at (9,0.00) {$2$};
\node[draw=blue,circle,scale=0.2, fill=blue!60] (B10) at (10,0.00) {};
\node[anchor=north,blue,scale=1.00] (LT10) at (10,0.00) {$3$};
\node[draw=blue,circle,scale=0.2, fill=blue!60] (B11) at (11,0.00) {};
\node[anchor=north,blue,scale=1.00] (LT11) at (11,0.00) {$4$};
\node[draw=red,circle,scale=0.3,fill=red] (B12) at (12,0.00) {};
\node[anchor=north,red,scale=1.00] (LT12) at (12,0.00) {$5$};
\node[draw=blue,circle,scale=0.3,fill=blue!60] (B13) at (13,0.00) {};
\node[anchor=north,blue,scale=1.00] (LT13) at (13,0.00) {$6$};
\draw[line width=0.5pt, blue!60] (T0)--(B3);
\draw[line width=1.3pt,red,rounded corners=\rc] (T1)--(4.4,1)--(B5);
\draw[line width=1.3pt,blue!60,rounded corners=\rc] (T3)--(4.6,1)--(B6);
\draw[line width=0.5pt, blue!60] (T2)--(B4);
\draw[line width=0.5pt, blue!60] (T4)--(B7);
\draw[line width=0.5pt, red] (T5)--(B8);
\draw[line width=0.5pt, blue!60] (T6)--(B9);
\draw[line width=0.5pt, blue!60] (T7)--(B10);
\draw[line width=0.5pt, blue!60] (T9)--(B11);
\draw[line width=1.3pt,red,rounded corners=\rc] (T8)--(11.4,1)--(B12);
\draw[line width=1.3pt,blue!60,rounded corners=\rc] (T10)--(11.6,1)--(B13);
\node[scale=1.50, anchor=west] (DDDTR) at (10.30, 3.00) {$\dots$};
\node[scale=1.50, anchor=east] (DDDTR) at (-0.30, 3.00) {$\dots$};
\node[scale=1.50, anchor=west] (DDDTR) at (13.30, 0.00) {$\dots$};
\node[scale=1.50, anchor=east] (DDDTR) at (2.70, 0.00) {$\dots$};
\end{tikzpicture}

}
& \hspace{0.4in} &
\scalebox{0.70}{
\begin{tikzpicture}[xscale=0.70, yscale=0.70,baseline=(ZUZU.base)]
\coordinate(ZUZU) at (0,1.50);
\draw[black!60,line width=0.3] (0,0) grid (4, 3);
\draw[black!70,line width=0.2] (0,0) -- (4, 3);
\node[draw=blue, line width=1pt, circle, scale=0.30] (A12) at (1, 1) {};
\node[draw=blue, line width=1pt, circle, scale=0.30] (A25) at (3, 2) {};
\node[anchor=east,blue] (K) at (0.00,1.50) {$3$};
\node[anchor=south,blue] (N) at (2.00,3.00) {$4$};
\end{tikzpicture}

}
\\
$\knk(f)=(3,4)$ & $\knk(\textcolor{red}{f_1^{(1,2)}})=(2,3)$ & $\knk(\textcolor{red}{f_1^{(1,3)}})=(1,1)$ & & $\Fr(f)=\{(1,1),(2,3)\}$
\end{tabular}
}

  \caption{Computing $\Fr(f)$ for $f\in\Bxc_{3,7}$.
} \label{fig:kn_def} 
\end{figure}
The permutation $\fbij\in S_n$ is a product of two cycles, say, $\fbij=(a_1a_2\cdots a_{n_1})(b_1b_2\cdots b_{n_2})$, where $a_1\equiv i$ and $b_1\equiv j$ modulo $n$. By taking order-preserving bijections $\{a_1,a_2,\dots,a_{n_1}\}\to[n_1]$ and $\{b_1,b_2,\dots,b_{n_2}\}\to[n_2]$, we may naturally view each of the two cycles as permutations $\fbij_1\in \Sc_{n_1}$ and $\fbij_2\in \Sc_{n_2}$. We thus have $\fij_1\in\Bxc_{k_1,n_1}$ and $\fij_2\in\Bxc_{k_2,n_2}$, where $k_1:=\k(\fbij_1)$ and $k_2:=\k(\fbij_2)$.  

\begin{definition}\label{dfn:F(f)}
For $f\in\Bknc$, the \emph{inversion multiset} $\Fr(f)$ contains a point $\knk(\fij_1)$ for each inversion $(i,j)$ of $f$. We say that $f$ is \emph{\repfree} if $\Fr(f)$ is actually a set, that is, if it contains exactly $\ell(f)$ distinct points.
\end{definition}
\noindent See \cref{fig:kn_def} for an example. When we draw the set $\Fr(f)$ inside a $k\times(n-k)$ rectangle, we swap the horizontal and vertical coordinates; cf. \cref{notn:swap}.

We have $\knk(\fij_1)+\knk(\fij_2)=(k,n-k)$, but note that we only include $\knk(\fij_1)$ in $\Fr(f)$ for each inversion $(i,j)$. Nevertheless, $\Fr(f)$ is always \emph{centrally symmetric}, that is, invariant under the map $\knk\mapsto (k,n-k)-\knk$; see \cref{cor:cs}.

\subsection{Main result}\label{sec:main-result}
For a set $\Flet\subset[k-1]\times[n-k-1]$, we let $\Fbi:=\Flet\sqcup\{(0,0),(k,n-k)\}$. We say that $\Flet$ is \emph{convex} if $\Fbi$ contains all lattice points of its convex hull. 
For $f\in \Bknc$ such that $\Fr(f)$ is convex, let $\DyckF(f)$ denote the set of lattice paths from $(0,0)$ to $(k,n-k)$ with up/right unit steps which stay above the main diagonal and avoid $\Fr(f)$. 
 For each $f\in\Bknc$, $\Fr(f)$ contains all lattice points $(k_1,n_1-k_1)\in [k-1]\times[n-k-1]$ that satisfy $k_1/n_1=k/n$ (\cref{slope:=}). Thus the paths in $\DyckF(f)$ always avoid the main diagonal.  %

\noindent
\begin{minipage}{\textwidth}
\begin{theorem}\ \label{thm:main}
  \begin{theoremlist}
  \item\label{thm:main:convex_counting} If $f\in\Bknc$ is \repfree then $\Fr(f)$ is centrally symmetric and convex, and 
\begin{equation}\label{eq:counting}
  \Cat_f=\#\DyckF(f).
\end{equation}
  \item\label{thm:main:existence} For any centrally symmetric convex subset $\Flet\subset[k-1]\times[n-k-1]$, there exists a \repfree $f\in\Bknc$ satisfying $\Fr(f)=\Flet$. 
  \end{theoremlist}
\end{theorem}
\end{minipage}

\begin{figure}
  \includegraphics[width=1.0\textwidth]{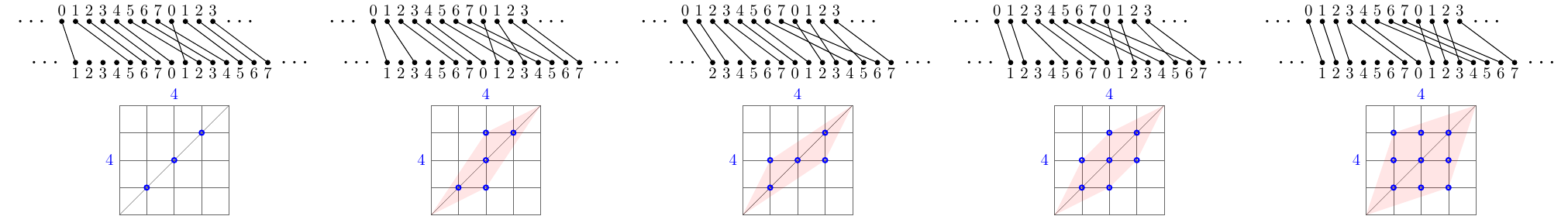}

  \caption[]{Top row: a collection of \repfree permutations $f\in\Bxc_{4,8}$ drawn in affine notation as in \cref{sec:perm-affine-notat}.  Bottom row: their inversion sets $\Fr(f)$; see \cref{dfn:F(f)} and \cref{ex:48}.}   \label{fig:Fs_48}
\end{figure}

\begin{example}\label{ex:48}
The bottom row of \cref{fig:Fs_48} contains all possible centrally symmetric convex subsets of $[k-1]\times[n-k-1]$ for $k=4$ and $n=8$. For each such subset $\Flet$, the top row contains a \repfree bounded affine permutation $f\in\Bknc$ satisfying $\Fr(f)=\Flet$.
\end{example}

\subsection{Other interpretations and further directions}
Even though our results are purely combinatorial, they provide a starting point for several unexpected connections to the recent results of~\cite{OR_Cox,GHSR,BHMPS} on the rational shuffle conjecture, Coxeter links, and flag Hilbert schemes. In particular, the appearance of convex sets in~\cite[Section~7]{BHMPS} indicates that open positroid varieties may provide the right geometric framework for the symmetric functions considered in~\cite{BHMPS}. We discuss these connections and list several conjectures in  \cref{sec:other}.

\subsection*{Acknowledgments}
We thank David Speyer and Eugene Gorsky for stimulating discussions.

\section{Bounded affine permutations}\label{sec:affine}

\subsection{Affine permutations}\label{sec:affine-permutations}
An ($n$-periodic) \emph{affine permutation} is a bijection $f:\Z\to\Z$ satisfying the periodicity condition $f(i+n)=f(i)+n$.  We let $\tSp_n$ denote the group (under composition) of $n$-periodic affine permutations.  Inversions, and the length function $\ell(f)$ (see \cref{sec:intro:rep-free}) are defined for any $f \in \tSp_n$.

For $k \in \Z$, let $\tSp_n^{(k)} \subset \tSp_n$ be the subset of affine permutations satisfying the condition
$$
\sum_{i=1}^n (f(i)-i) = kn.
$$
Then $\tSp_n = \bigsqcup_{k\in \Z} \tSp_n^{(k)}$.  The subgroup $\tSp_n^{(0)}$ is the Coxeter group of affine type $A$.  The group $\tSp_n$ is usually called the \emph{extended affine Weyl group}.

A \emph{bounded affine permutation} is an affine permutation $f \in \tSp_n$ that satisfies the additional condition $i \leq f(i) \leq i+n$ for all $i\in\Z$.  Denote by $\Bkn$ the (finite) set of bounded affine permutations in $\tSp_n^{(k)}$, called the set of \emph{$(k,n)$-bounded affine permutations}.  We see that if $\perm\in\Snc$ and $k = \k(\perm)$ then the associated bounded affine permutation $f$ (cf. \cref{sec:perm-affine-notat}) belongs to $\Bkn$. In other words, we have $\Bknc \subset \Bkn$.

For $k \in \Z$, let $\fkn \in \tSp_n^{(k)} \subset \tSp_n$ be given by $i\mapsto i+k$ for all $i\in\Z$.  Then $\{\fkn\mid k\in\Z\}$ is exactly the set of length $0$ elements in $\tSp_n$, and for $0 \leq k \leq n$ we have $\fkn \in \Bkn$.

For $i\in\Z$, let $s_i \in \tSp_n$ be the simple transposition given by $i\mapsto i+1$, $i+1\mapsto i$, and $j\mapsto j$ for all $j\not\equiv i,i+1$ modulo $n$. For $f\in\Bkn$ and $i\in\Z$, we have $\ell(s_if)=\ell(f)\pm1$ and $\ell(fs_i)=\ell(f)\pm1$. We write $fs_i<f$ if $\ell(fs_i)<\ell(f)$, and similarly for $fs_i>f$, $s_if<f$, and $s_if>f$.

Given $f\in\tSp_n$, define the \emph{cyclic shift} $\shf \in\tSp_n$ by
\begin{equation}\label{eq:cyc_shift}
  (\shf)(i):=f(i-1)+1 \quad\text{for all $i\in\Z$.}
\end{equation}
In other words, we have $\shf= f_{1,n} f f_{1,n}^{-1}$. Note that $\sigma$ preserves each of the subsets $\tSp_n^{(k)}$, $\Bkn$, and $\Bknc$.

\subsection{Conjugation and double move reduction}\label{sec:conj-double-move}

\begin{definition}
We say that $f\in\Bkn$ has a \emph{double crossing} at some $i\in\Z$ if $s_ifs_i<s_if<f$, $s_ifs_i<fs_i<f$, and $s_ifs_i,s_if,fs_i\in\Bkn$.
\end{definition}
Equivalently, for $a:=f^{-1}(i+1)$, $b:=f^{-1}(i)$, $c:=f(i+1)$, $d:=f(i)$, $f$ has a double crossing at $i$ if and only if $a<b<i<i+1<c<d$. See \figref{fig:moves}(right). In this case, we say that $s_ifs_i$ is obtained from $f$ by a \emph{double move}.

\begin{definition}
Let $f\in \Bkn$, $i\in\Z$, and $f':=s_ifs_i$. If $\ell(f)=\ell(s_ifs_i)$ and $s_ifs_i\in\Bkn$ then we say that $f$ and $f'$ are related by a \emph{length-preserving simple conjugation}. We say that $f,g\in\Bkn$ are \emph{\Ceqvt} and write $f\ceq g$ if $f$ and $g$ can be related by a sequence of length-preserving simple conjugations. See \figref{fig:moves}(left).
\end{definition}

\begin{figure}
\resizebox{\textwidth}{!}{

\setlength{\tabcolsep}{-2pt}
\begin{tabular}{cccc|cccc}
\scalebox{0.70}{\begin{tikzpicture}[xscale=0.40, yscale=0.40, baseline=(ZUZU.base)]
	\coordinate(ZUZU) at (0,2.00);
\node[draw,circle,scale=0.30,fill=black] (T0) at (0,4.00) {};
\node[anchor=south,scale=1.00] (LT0) at (0,4.00) {$a\strut$};
\node[draw,circle,scale=0.30,fill=black] (T1) at (1,4.00) {};
\node[draw,circle,scale=0.30,fill=black] (T2) at (2,4.00) {};
\node[draw,circle,scale=0.30,fill=black] (T3) at (3,4.00) {};
\node[anchor=south,scale=1.00] (LT3) at (3,4.00) {$b\strut$};
\node[draw,circle,scale=0.30,fill=black] (T4) at (4,4.00) {};
\node[draw,circle,scale=0.30,fill=black] (T5) at (5,4.00) {};
\node[anchor=south,scale=1.00] (LT5) at (5,4.00) {$i\strut\ \ $};
\node[draw,circle,scale=0.30,fill=black] (T6) at (6,4.00) {};
\node[anchor=south,scale=1.00] (LT6) at (6,4.00) {$\quad i+1\strut$};
\node[draw,circle,scale=0.30,fill=black] (B5) at (5,0.00) {};
\node[anchor=north,scale=1.00] (LT5) at (5,0.00) {$i\strut\ \ $};
\node[draw,circle,scale=0.30,fill=black] (B6) at (6,0.00) {};
\node[anchor=north,scale=1.00] (LT6) at (6,0.00) {$\quad i+1\strut$};
\node[draw,circle,scale=0.30,fill=black] (B7) at (7,0.00) {};
\node[draw,circle,scale=0.30,fill=black] (B8) at (8,0.00) {};
\node[draw,circle,scale=0.30,fill=black] (B9) at (9,0.00) {};
\node[anchor=north,scale=1.00] (LT9) at (9,0.00) {$c\strut$};
\node[draw,circle,scale=0.30,fill=black] (B10) at (10,0.00) {};
\node[draw,circle,scale=0.30,fill=black] (B11) at (11,0.00) {};
\node[anchor=north,scale=1.00] (LT11) at (11,0.00) {$d\strut$};
\draw[black, line width=1.00pt] (T0)--(B5);
\draw[black, line width=1.00pt] (T3)--(B6);
\draw[black, line width=1.00pt] (T5)--(B11);
\draw[black, line width=1.00pt] (T6)--(B9);
\node[scale=1.50, anchor=west] (DDDTR) at (6.30, 4.00) {$\dots$};
\node[scale=1.50, anchor=east] (DDDTR) at (-0.30, 4.00) {$\dots$};
\node[scale=1.50, anchor=west] (DDDTR) at (11.30, 0.00) {$\dots$};
\node[scale=1.50, anchor=east] (DDDTR) at (4.70, 0.00) {$\dots$};
\end{tikzpicture}
} & $\longleftrightarrow$ & \scalebox{0.70}{\begin{tikzpicture}[xscale=0.40, yscale=0.40, baseline=(ZUZU.base)]
	\coordinate(ZUZU) at (0,2.00);
\node[draw,circle,scale=0.30,fill=black] (T0) at (0,4.00) {};
\node[anchor=south,scale=1.00] (LT0) at (0,4.00) {$a\strut$};
\node[draw,circle,scale=0.30,fill=black] (T1) at (1,4.00) {};
\node[draw,circle,scale=0.30,fill=black] (T2) at (2,4.00) {};
\node[draw,circle,scale=0.30,fill=black] (T3) at (3,4.00) {};
\node[anchor=south,scale=1.00] (LT3) at (3,4.00) {$b\strut$};
\node[draw,circle,scale=0.30,fill=black] (T4) at (4,4.00) {};
\node[draw,circle,scale=0.30,fill=black] (T5) at (5,4.00) {};
\node[anchor=south,scale=1.00] (LT5) at (5,4.00) {$i\strut\ \ $};
\node[draw,circle,scale=0.30,fill=black] (T6) at (6,4.00) {};
\node[anchor=south,scale=1.00] (LT6) at (6,4.00) {$\quad i+1\strut$};
\node[draw,circle,scale=0.30,fill=black] (B5) at (5,0.00) {};
\node[anchor=north,scale=1.00] (LT5) at (5,0.00) {$i\strut\ \ $};
\node[draw,circle,scale=0.30,fill=black] (B6) at (6,0.00) {};
\node[anchor=north,scale=1.00] (LT6) at (6,0.00) {$\quad i+1\strut$};
\node[draw,circle,scale=0.30,fill=black] (B7) at (7,0.00) {};
\node[draw,circle,scale=0.30,fill=black] (B8) at (8,0.00) {};
\node[draw,circle,scale=0.30,fill=black] (B9) at (9,0.00) {};
\node[anchor=north,scale=1.00] (LT9) at (9,0.00) {$c\strut$};
\node[draw,circle,scale=0.30,fill=black] (B10) at (10,0.00) {};
\node[draw,circle,scale=0.30,fill=black] (B11) at (11,0.00) {};
\node[anchor=north,scale=1.00] (LT11) at (11,0.00) {$d\strut$};
\draw[black, line width=1.00pt] (T0)--(B6);
\draw[black, line width=1.00pt] (T3)--(B5);
\draw[black, line width=1.00pt] (T5)--(B9);
\draw[black, line width=1.00pt] (T6)--(B11);
\node[scale=1.50, anchor=west] (DDDTR) at (6.30, 4.00) {$\dots$};
\node[scale=1.50, anchor=east] (DDDTR) at (-0.30, 4.00) {$\dots$};
\node[scale=1.50, anchor=west] (DDDTR) at (11.30, 0.00) {$\dots$};
\node[scale=1.50, anchor=east] (DDDTR) at (4.70, 0.00) {$\dots$};
\end{tikzpicture}
} & \hspace{0.2in} &\hspace{0.2in} & \scalebox{0.70}{\begin{tikzpicture}[xscale=0.40, yscale=0.40, baseline=(ZUZU.base)]
	\coordinate(ZUZU) at (0,2.00);
\node[draw,circle,scale=0.30,fill=black] (T0) at (0,4.00) {};
\node[anchor=south,scale=1.00] (LT0) at (0,4.00) {$a\strut$};
\node[draw,circle,scale=0.30,fill=black] (T1) at (1,4.00) {};
\node[draw,circle,scale=0.30,fill=black] (T2) at (2,4.00) {};
\node[draw,circle,scale=0.30,fill=black] (T3) at (3,4.00) {};
\node[anchor=south,scale=1.00] (LT3) at (3,4.00) {$b\strut$};
\node[draw,circle,scale=0.30,fill=black] (T4) at (4,4.00) {};
\node[draw,circle,scale=0.30,fill=black] (T5) at (5,4.00) {};
\node[anchor=south,scale=1.00] (LT5) at (5,4.00) {$i\strut\ \ $};
\node[draw,circle,scale=0.30,fill=black] (T6) at (6,4.00) {};
\node[anchor=south,scale=1.00] (LT6) at (6,4.00) {$\quad i+1\strut$};
\node[draw,circle,scale=0.30,fill=black] (B5) at (5,0.00) {};
\node[anchor=north,scale=1.00] (LT5) at (5,0.00) {$i\strut\ \ $};
\node[draw,circle,scale=0.30,fill=black] (B6) at (6,0.00) {};
\node[anchor=north,scale=1.00] (LT6) at (6,0.00) {$\quad i+1\strut$};
\node[draw,circle,scale=0.30,fill=black] (B7) at (7,0.00) {};
\node[draw,circle,scale=0.30,fill=black] (B8) at (8,0.00) {};
\node[draw,circle,scale=0.30,fill=black] (B9) at (9,0.00) {};
\node[anchor=north,scale=1.00] (LT9) at (9,0.00) {$c\strut$};
\node[draw,circle,scale=0.30,fill=black] (B10) at (10,0.00) {};
\node[draw,circle,scale=0.30,fill=black] (B11) at (11,0.00) {};
\node[anchor=north,scale=1.00] (LT11) at (11,0.00) {$d\strut$};
\draw[black, line width=1.00pt] (T0)--(B6);
\draw[black, line width=1.00pt] (T3)--(B5);
\draw[black, line width=1.00pt] (T5)--(B11);
\draw[black, line width=1.00pt] (T6)--(B9);
\node[scale=1.50, anchor=west] (DDDTR) at (6.30, 4.00) {$\dots$};
\node[scale=1.50, anchor=east] (DDDTR) at (-0.30, 4.00) {$\dots$};
\node[scale=1.50, anchor=west] (DDDTR) at (11.30, 0.00) {$\dots$};
\node[scale=1.50, anchor=east] (DDDTR) at (4.70, 0.00) {$\dots$};
\end{tikzpicture}
} & $\longrightarrow$ & \scalebox{0.70}{\begin{tikzpicture}[xscale=0.40, yscale=0.40, baseline=(ZUZU.base)]
	\coordinate(ZUZU) at (0,2.00);
\node[draw,circle,scale=0.30,fill=black] (T0) at (0,4.00) {};
\node[anchor=south,scale=1.00] (LT0) at (0,4.00) {$a\strut$};
\node[draw,circle,scale=0.30,fill=black] (T1) at (1,4.00) {};
\node[draw,circle,scale=0.30,fill=black] (T2) at (2,4.00) {};
\node[draw,circle,scale=0.30,fill=black] (T3) at (3,4.00) {};
\node[anchor=south,scale=1.00] (LT3) at (3,4.00) {$b\strut$};
\node[draw,circle,scale=0.30,fill=black] (T4) at (4,4.00) {};
\node[draw,circle,scale=0.30,fill=black] (T5) at (5,4.00) {};
\node[anchor=south,scale=1.00] (LT5) at (5,4.00) {$i\strut\ \ $};
\node[draw,circle,scale=0.30,fill=black] (T6) at (6,4.00) {};
\node[anchor=south,scale=1.00] (LT6) at (6,4.00) {$\quad i+1\strut$};
\node[draw,circle,scale=0.30,fill=black] (B5) at (5,0.00) {};
\node[anchor=north,scale=1.00] (LT5) at (5,0.00) {$i\strut\ \ $};
\node[draw,circle,scale=0.30,fill=black] (B6) at (6,0.00) {};
\node[anchor=north,scale=1.00] (LT6) at (6,0.00) {$\quad i+1\strut$};
\node[draw,circle,scale=0.30,fill=black] (B7) at (7,0.00) {};
\node[draw,circle,scale=0.30,fill=black] (B8) at (8,0.00) {};
\node[draw,circle,scale=0.30,fill=black] (B9) at (9,0.00) {};
\node[anchor=north,scale=1.00] (LT9) at (9,0.00) {$c\strut$};
\node[draw,circle,scale=0.30,fill=black] (B10) at (10,0.00) {};
\node[draw,circle,scale=0.30,fill=black] (B11) at (11,0.00) {};
\node[anchor=north,scale=1.00] (LT11) at (11,0.00) {$d\strut$};
\draw[black, line width=1.00pt] (T0)--(B5);
\draw[black, line width=1.00pt] (T3)--(B6);
\draw[black, line width=1.00pt] (T5)--(B9);
\draw[black, line width=1.00pt] (T6)--(B11);
\node[scale=1.50, anchor=west] (DDDTR) at (6.30, 4.00) {$\dots$};
\node[scale=1.50, anchor=east] (DDDTR) at (-0.30, 4.00) {$\dots$};
\node[scale=1.50, anchor=west] (DDDTR) at (11.30, 0.00) {$\dots$};
\node[scale=1.50, anchor=east] (DDDTR) at (4.70, 0.00) {$\dots$};
\end{tikzpicture}
}
\\ &&&&\vspace{-0.1in}&&& \\
$f$ && $s_ifs_i$ &&& $f$ && $s_ifs_i$ \\
\multicolumn{3}{c}{length-preserving simple conjugation} & & & \multicolumn{3}{c}{double move}
\end{tabular}
}
  \caption{\label{fig:moves} Moves for computing $\Rt_f(q)$ and $\Cat_f$.}
\end{figure}
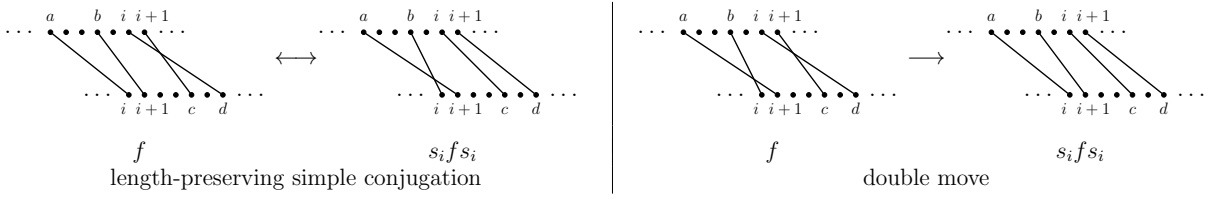

The following result describes the structure of $\Bknc$ under double moves and \Ceqvce.
\begin{proposition}\label{prop:structure} \
\begin{theoremlist}
\item\label{item:minimumlength} The minimal length elements of $\Bknc$ are of length $d:=\gcd(k,n)-1$ and all such elements are related by cyclic shift~\eqref{eq:cyc_shift} and \Ceqvce.
\item\label{item:reduce} Any $f \in \Bknc$ can be reduced to a minimal length element of $\Bknc$ by double moves and \Ceqvce.
\end{theoremlist}
\end{proposition}

\subsection{Proof of Proposition~\ref{prop:structure}}
We deduce these statements from the results of He and Nie \cite{HN} and He and Yang \cite{HY}.

Following~\cite{HN}, we introduce the following notation. For $f,f' \in \tSp_n$, we write $f \to f'$ if there is a sequence $f = f_0, f_1,f_2,\ldots,f_r = f'$ such that $f_j = s_{i_j} f_{j-1} s_{i_{j}}$ for $j=1,2,\ldots,r$, satisfying $\ell(f_j) \leq \ell(f_{j-1})$.  We write $f \capprox f'$ if $f \to f'$ and $f' \to f$.  Thus $f \to f'$ if $f'$ can be obtained from $f$ by a sequence of \Ceqvces and double moves, and $f \capprox f'$ if $f$ and $f'$ are \Ceqvt, without the restriction on staying inside $\Bkn$.

\begin{lemma}\label{lem:inside}
Let $f \in \Bknc$ and $f'\in\tSp_n$ be such that $f \to f'$.  Then $f' \in \Bknc$.
\end{lemma}
\begin{proof}
Suppose $f \in \Bknc$ and $f' = s_i f s_i$ satisfies $\ell(f') \leq \ell(f)$.  Since $f \in \Bknc$, we have $f(j) \in [j+1,j+n-1]$ for all $j \in \Z$.  It follows from this that $f'$ is also a bounded affine permutation, and thus $f' \in \Bknc$.
\end{proof}

\begin{theorem}[{\cite[Theorem 2.9]{HN}}] \label{thm:HN}
Let $\O'$ be an $\tSp^{(0)}_n$-conjugacy class in $\tSp_n$ and let $\Omin' \subset \O'$ denote the set of elements of minimal length.  Then for any $f \in \O'$, there exists $f' \in \Omin'$ such that $f \to f'$.
\end{theorem}

\begin{proposition}\label{prop:HN}
Suppose that $\O'$ is an $\tSp^{(0)}_n$-conjugacy class in $\tSp_n$ with a nonempty intersection with $\Bknc$.  
Then for $f,f' \in \Omin'$, we have $f \capprox f'$.
\end{proposition}
\begin{proof}
For $f \in \Bknc$, the image $\perm \in S_n$ is an $n$-cycle.  The $n$-cycles are {\it elliptic} elements in $S_n$, so by \cite[Corollary 4.7]{HN}, we have that $\O'$ is {\it nice} in the sense of \cite[Section~4.1]{HN}. It follows from the definition of \emph{nice} that $f \capprox f'$ for $f,f' \in \Omin'$. 
\end{proof}

\begin{proposition}\label{prop:allone}
The elements of $\Bknc$ all belong to a single $\tSp_n$-conjugacy class $\O$ in $\tSp_n$.
\end{proposition}
\begin{proof}
Our goal is to apply~\cite[Proposition~2.1]{HY}.  In the notation of~\cite{HY}, we have $\delta=\id$, $\tilde W'=\tilde W$, $(P^\vee/Q^\vee)_\delta\cong \Z/n\Z$, and $\Ocal_0$ is the $S_n$-conjugacy class consisting of $n$-cycles in $S_n$. Choosing $\nu:=k\in \Z/n\Z$, we see that $\kappa^{-1}_\delta(\nu)$ contains $\Bknc$. Similarly, $\eta^{-1}(\Ocal_0)$ contains all bounded affine permutations whose reduction modulo $n$ is  an $n$-cycle. Thus $\Bknc$ is a subset of $\eta^{-1}(\Ocal_0)\cap \kappa^{-1}_\delta(\nu)$, which, according to~\cite[Proposition~2.1]{HY}, is a single $\tSp_n$-conjugacy class in $\tSp_n$.
\end{proof}

We are ready to finish the proof of \cref{prop:structure}. By Proposition~\ref{prop:allone}, there is an $\tSp_n$-conjugacy class $\O \subset \tSp_n$ containing $\Bknc$.  Since $\tSp_n = \tSp^{(0)}_n \rtimes \langle f_{1,n} \rangle$, there exist finitely-many distinct $\tSp^{(0)}_n$-conjugacy classes $\O'_0,\ldots,\O'_{r-1}$ such that $\O = \bigsqcup_i \O'_i$ and $\O'_i = \sigma^i \O'_0$.  Note that the cyclic shift $\sigma$ is length-preserving.  Thus the minimal length elements in $\O'_0,\ldots,\O'_{r-1}$ have the same length, and this length is equal to the minimal length of any element in $\Bknc$. 

It is easy to see that $\ncyc(f_{k,n}) = \gcd(k,n)$, where for $f \in \tSp_n$, we denote by $\ncyc(f):=\ncyc(\fb)$ the number of cycles of the permutation $\fb$.  Now, for $f \in \tSp_n$, we have $\ncyc(s_i f) = \ncyc(f s_i) \in \{\ncyc(f)+1,\ncyc(f) -1\}$.  It follows that for $f \in \Bknc$, we have $\ell(f) \geq \gcd(k,n)-1$.  On the other hand, it is easy to see that $f_{k,n} s_1 s_2 \cdots s_{ \gcd(k,n)-1} \in \Bknc$.  Thus the minimal length of $f \in \Bknc$ is $d:=\gcd(k,n)-1$. Since $\ncyc(\cdot)$ is invariant under conjugation, we find that any $f\in\Bknc$ with $\ell(f)=d$ has minimal length in its  $\tSp^{(0)}_n$-conjugacy class.

Let $f,f'\in \Bknc$ be two elements of length $d$. By \cref{prop:allone}, $f$ and $f'$ are $\tSp_n$-conjugate. Thus $f$ is $\tSp^{(0)}_n$-conjugate to a cyclic shift $g=\sigma^if'$ of $f'$. Let $\O'$ be the $\tSp^{(0)}_n$-conjugacy class containing $f$ and $g$. Since $\ell(f)=\ell(g)=d$ is minimal, by \cref{prop:HN}, we get $f\capprox g$. By \cref{lem:inside}, having $f\capprox g$ for $f\in\Bknc$ implies that $f\ceq g$. This proves \cref{item:minimumlength}.

As we showed above, the minimum length elements of $\Bknc$ are also minimum length elements in their $\tSp_n^{(0)}$-conjugacy class.  Thus \cref{item:reduce} follows from Theorem~\ref{thm:HN} combined with \cref{lem:inside}. \qed

\section{Positroid Catalan numbers}\label{sec:affine_R_poly}

\subsection{$R$-polynomials}
For each $(k,n)$-bounded affine permutation $f$, let $\Pio_f$ denote the \emph{open positroid variety}~\cite{KLS}.  The \emph{$R$-polynomial} $R_f(q):=\#\Pio_f(\Fbb_q)$ counts the number of points in $\Pio_f$ over a finite field $\Fbb_q$ with $q$ elements (where $q$ is a prime power). These $R$-polynomials are special cases of the $R$-polynomials of Kazhdan and Lusztig~\cite{KL1,KL2}.  %

The following recurrence appears in~\cite[Section~4]{MSLA}. 
\begin{proposition}\label{prop:MSrec}
The polynomials $R_f(q)$, $f\in\Bkn$, may be computed from the following recurrence. %
\begin{theoremlist}[label=\normalfont(\alph*)]
\item\label{item:trivial} If $n=1$ then $R_f(q)=1$.
\item\label{item:MS_fixedpt} If $\fb$ has some fixed points then $R_f(q)=R_{f'}(q)$, where $f'$ is obtained from $f$ by removing all fixed points of $\fb$. %
\item\label{item:MS_i_i+1} If $f(i)=i+1$ or $f(i+1)=i+n$ (where $n\geq2$) then $fs_i, s_i f\in \Bkn$ and $R_f(q)=(q-1)R_{s_if}(q)=(q-1)R_{fs_i}(q)$. 
\item\label{item:ceq} If $f\ceq g$ then $R_f(q)=R_g(q)$.
\item\label{item:doublecross} If $f$ has a double crossing at $i\in\Z$ then 
\begin{equation}\label{eq:MSrec:double_cross}
  R_{s_ifs_i}(q)=(q-1)R_{s_if}(q)+qR_f(q).
\end{equation}
\end{theoremlist}
\end{proposition}
\begin{proof}
The results of \cite{MSLA} are formulated in the language of cluster algebras. For the convenience of the reader, we give an alternative proof of \itemref{item:trivial}--\itemref{item:doublecross} not relying on cluster algebras,  assuming familiarity with \cite{KLS}.  We start by noting that the definition $R_f(q):=\#\Pio_f(\Fbb_q)$ implies that $R_f(q)=R_{\sigma f}(q)$ since the open positroid varieties indexed by $f$ and by $\shf$ are isomorphic.

The initial condition~\itemref{item:trivial} is trivial.  Property~\itemref{item:MS_fixedpt} %
follows from the definition of the open positroid variety $\Pio_f$.  If $f(i) = i$ (resp., $f(i) = i+n$) then $\Pio_f$ maps isomorphically to another open positroid variety $\Pio_{f'}$ under the natural projection map $\Gr(k,n) \to \Gr(k,n-1)$ (resp., $\Gr(k,n) \to \Gr(k-1,n-1)$) between Grassmannians that removes (resp., contracts) the $i$-th column. See e.g.~\cite[Lemmas~7.8 and~7.9]{LamCDM_tnn}.

The Kazhdan--Lusztig $R$-polynomials $\KLR_v^w$ are indexed by pairs $(v,w)$ of permutations.  When $v\not\leq w$ (where $\leq$ denotes the Bruhat order on $S_n$), we have $\KLR_v^w=0$, and for $v=w$, we have $\KLR_v^w=1$. For $v\leq w\in S_n$, $\KLR_v^w$ can then be computed by a recurrence relation~\cite[Section~2]{KL1}:
\begin{equation}\label{eq:KLR_dfn}
\KLR_v^w=
  \begin{cases}
    \KLR_{sv}^{sw}, &\text{if $sv<v$ and $sw<w$,}\\
    (q-1)\KLR_{sv}^w+q\KLR_{sv}^{sw}, &\text{if $sv>v$ and $sw<w$.}\\
  \end{cases}\\
\end{equation}
Here, $s=s_i$ for some $1\leq i\leq n-1$ is a simple transposition satisfying $sw<w$. 

For each $f \in \Bkn$, there is a pair $(v,w)$ such that $f = w \tau_{k,n} v^{-1}$  and $R_f(q) = \KLR_v^w$, where $\tau_{k,n}\in\tS_n$ denotes a certain \emph{translation element}; see~\cite[Proposition~3.15]{KLS}.  From this, \eqref{eq:KLR_dfn} implies \itemref{item:ceq}--\itemref{item:doublecross} whenever we have a length-preserving simple conjugation or a double crossing at $1 \leq i \leq n-1$.  Applying the cyclic shift, we see that properties~\itemref{item:ceq}--\itemref{item:doublecross} hold also for $i = 0$, which completes their proof.

Property~\itemref{item:MS_i_i+1} for $1 \leq i \leq n-1$ also follows from \eqref{eq:KLR_dfn}.  If $f(i)=i+1$ or $f(i+1)=i+n$ then $\ell(s_if) = \ell(fs_i) = \ell(f) + 1$, and this corresponds to the case $\KLR_v^w = (q-1)\KLR_{sv}^w$; here, $\KLR_{sv}^{sw} = 0$ since $sv \not \leq sw$.  In the remaining case $i=0$,  \itemref{item:MS_i_i+1} follows from applying the cyclic shift.

Finally, a constructive algorithm to compute $R_f(q)$ from~\itemref{item:trivial}--\itemref{item:doublecross} is given in the proof of~\cite[Theorem~3.3]{MSLA}.
\end{proof}

\subsection{Positroid Catalan numbers}\label{sec:pos_cat_nums}
Recall that for a permutation $\perm\in \Sn$, we let $\ncyc(\perm)=\ncyc(f)$ denote its number of cycles.  For $f\in\Bkn$, we let 
\begin{equation}\label{eq:Rt_dfn}
  \Rt_f(q):=R_f(q)/(q-1)^{n-\ncyc(f)}.
\end{equation}
It is easy to see (for example using~\eqref{eq:MSrec:double_cross_Rt} below; see also~\cite[Proposition~4.5]{qtcat}) that $\Rt_f(q)$ is always a polynomial in $q$.

The definition of a positroid Catalan number $\Cat_f$ (Definition~\ref{defn:intro}) can be extended to all $f \in \Bkn$ by setting
\begin{equation}\label{eq:Catf}
  \Cat_f:=\Rt_f(1).
\end{equation}
The relation to Definition~\ref{defn:intro} is given in~\cref{sec:Euler}. See \cref{fig:Rt} for examples. 

\begin{figure}
  \includegraphics[width=1.0\textwidth]{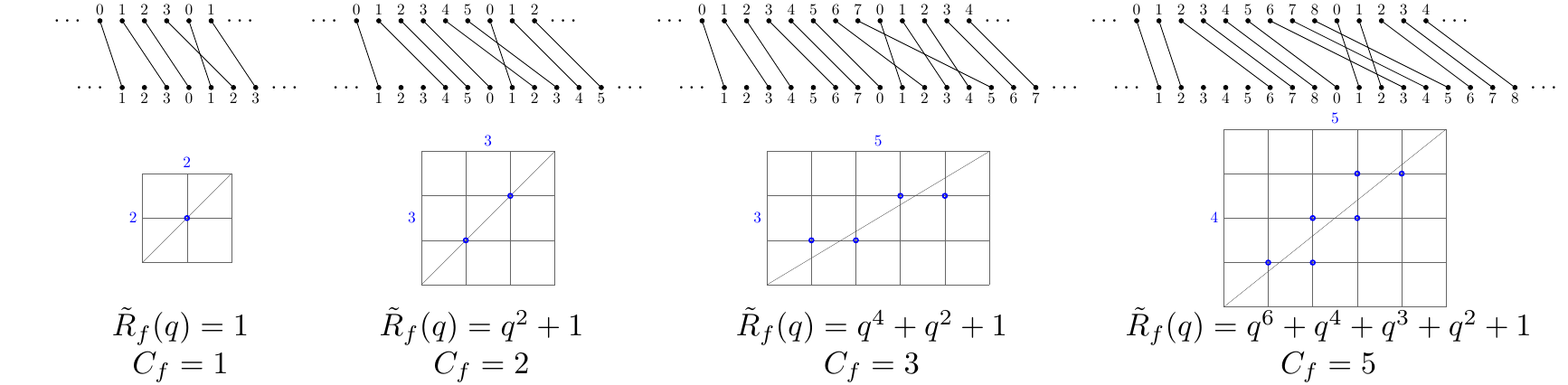}
  \caption{\label{fig:Rt} Some examples of $\Rt_f(q)$ and $\Cat_f$.}
\end{figure}

\subsection{Recurrence for positroid Catalan numbers}\label{sec:recurs-positr-catal}
If $f$ has a double crossing at $i\in\Z$ then~\eqref{eq:MSrec:double_cross} implies
\begin{equation}\label{eq:MSrec:double_cross_Rt}
  \Rt_{s_ifs_i}(q)=
  \begin{cases}
    \Rt_{s_if}(q)+q\Rt_f(q), &\text{if $i,i+1$ belong to the same cycle of $\fb$;}\\
    (q-1)^2\Rt_{s_if}(q)+q\Rt_f(q), &\text{if $i,i+1$ belong to different cycles of $\fb$.}\\
  \end{cases}
\end{equation}
Here, $i,i+1$ are considered modulo $n$.   

The next result follows from~\cref{prop:MSrec} combined with~\eqref{eq:Catf}--\eqref{eq:MSrec:double_cross_Rt}.
\begin{proposition}\label{prop:Cat_rec} The positroid Catalan numbers $\Cat_f$, $f\in\Bkn$, may be computed from the following recurrence.
\begin{theoremlist}[label=\normalfont(\alph*$'$)]
\item\label{item:Cat_trivial} If $n=1$ then $\Cat_f=1$.
\item  \label{item:Cat_MS_fixedpt} If $\fb$ has some fixed points then $\Cat_f=\Cat_{f'}$, where $f'$ is obtained from $f$ by removing all fixed points of $\fb$. %
\item\label{item:Cat_MS_i_i+1} If $f(i)=i+1$ or $f(i+1)=i+n$ (where $n\geq2$) then $fs_i, s_i f\in \Bkn$ and $\Cat_f=\Cat_{s_if}=\Cat_{fs_i}$.
\item If $f\ceq g$  then $\Cat_f=\Cat_g$.
\item\label{item:Cat_doublecross} If $f\in\Bkn$ has a double crossing at $i\in\Z$ then 
\begin{equation}\label{eq:Cat_rec}
    \Cat_{s_ifs_i}=
  \begin{cases}
    \Cat_{s_if}+\Cat_f, &\text{if $i,i+1$ belong to the same cycle of $\fb$;}\\
    \Cat_f, &\text{if $i,i+1$ belong to different cycles of $\fb$.}\\
  \end{cases}
\end{equation}
\end{theoremlist}
\end{proposition}

\begin{proposition}
Let $f\in\Bkn$. Then $\Cat_f$ is a positive integer. 
\end{proposition}
\begin{proof}
The proof of~\cite[Theorem~3.3]{MSLA} shows that $\Cat_f$ may be expressed using~\itemref{item:Cat_trivial}--\itemref{item:Cat_doublecross} in terms of $\Cat_g$ for bounded affine permutations $g$ satisfying either $\n(g)<\n(f)$ or $\n(g)=\n(f)$ and $\ell(g)>\ell(f)$. In particular, the recurrence in \cref{prop:Cat_rec} is subtraction-free, which shows the result. See also~\cite[Remark~9.4 and Proposition~9.5]{qtcat}.
\end{proof}
\begin{remark}
It is not always true that $\Rt_f(q)$ has positive coefficients: see~\cite[Example~4.22]{qtcat}. This question is closely related to the \emph{odd cohomology vanishing} phenomenon which appears for $\gcd(k,n)=1$ and $f=\fkn$ (i.e., for torus knots) but not for all $f\in\Bknc$. 
It is an important open problem to describe a wider class of positroids (or more generally, knots) for which this phenomenon occurs. We expect this class to contain all $f\in\Bknc$ which are \repfree; see \cref{conj:GHSR}.
\end{remark}

Let $f\in\Bkn$ be such that $\fb=(a^\parr1_1\cdots a^\parr1_{n_1})(a^\parr2_1\cdots a^\parr2_{n_2})\cdots(a^\parr r_1\cdots a^\parr r_{n_r})$ is a product of $r$ cycles. (The case $r=2$ was considered in \cref{sec:intro:rep-free}.) For each $j\in[r]$, denote by $f|_{S_j}\in\BND(k_j,n_j)$ the restriction of $f$ to the set $S_j$ of all integers congruent to one of $a^\parr j_1,\dots,a^\parr j_{n_j}$ modulo $n$. We deduce the following \emph{decoupling property} from \cref{prop:Cat_rec}.

\begin{corollary}[Decoupling]\label{cor:decoupling}
Let $f\in\Bkn$ and $i\in\Z$. If $i$ and $i+1$ belong to different cycles of $\fb$ then $s_ifs_i\in\Bkn$ and
\begin{equation}\label{eq:Cat_f=Cat_sifsi}
  \Cat_f=\Cat_{s_ifs_i}.
\end{equation}
In particular, in the above notation, for $f\in\Bkn$ we have
\begin{equation}\label{eq:Cat_f=product}
  \Cat_f=\prod_{j=1}^r \Cat_{f|_{S_j}}.
\end{equation}
\end{corollary}
\begin{proof}
Eq.~\eqref{eq:Cat_f=Cat_sifsi} follows easily from \cref{prop:Cat_rec}. 
To deduce~\eqref{eq:Cat_f=product}, we apply~\eqref{eq:Cat_f=Cat_sifsi} repeatedly until each cycle of $\fb$ is supported on a cyclically consecutive interval $[a,b]\subset [n]$ for some $a,b\in[n]$. After that, $\Cat_f$ may be computed via \cref{prop:Cat_rec} independently on each interval, which results in the product formula~\eqref{eq:Cat_f=product}.
\end{proof}

We will use a special case of~\eqref{eq:Cat_f=product} when $r=2$.   

\begin{corollary}\label{cor:Cat_rec}
Suppose that $f\in\Bknc$ has a double crossing at $i\in[n]$. Then
\begin{equation}\label{eq:recurrence_early}
  \Cat_{s_ifs_i}=\Cat_{\fii_1}\Cat_{\fii_2}+\Cat_f.
\end{equation}
\end{corollary}
\noindent Our eventual goal will be to relate~\eqref{eq:recurrence_early} to the recurrence for Dyck paths shown in \cref{fig:recurrence}.  One other simple result we will need is the cyclic shift invariance of $\Cat_f$ and $\Fr(f)$. 
\begin{proposition}\label{prop:cyc_shift_iso}
For any $f\in\Bkn$, we have 
\begin{equation*}%
 \Fr(f)=\Fr(\shf),\quad \Cat_f=\Cat_{\shf}, \quad\text{and}\quad \Rt_f(q)=\Rt_{\shf}(q).
\end{equation*}
\end{proposition}
\begin{proof}
It is obvious that both the definition of $\Fr(f)$ and the recurrence in~\cref{prop:MSrec,prop:Cat_rec} are invariant under the action of $\sigma$.
\end{proof}

\section{Big paths}\label{sec:big-paths}
The next few sections contain the main body of the proof of \cref{thm:main}. From now on, we switch from working in the $(k,n-k)$-coordinates to working in the $(k,n)$-coordinates. For $f\in\Bknc$, we let 
\begin{equation}\label{eq:del_dfn}
  \kn(f):=(k,n)
\end{equation}
 and define the multiset $\F(f)$ to be the image of $\Fr(f)$ under the map $(k_1,n_1-k_1)\mapsto (k_1,n_1)$. We let $\Fbf:=\F(f)\sqcup\{(0,0),(k,n)\}.$

Let $f\in\Bknc$. Our goal is to give a geometric interpretation of the multiset $\F(f)$.

\begin{notation}\label{notn:swap}
When referring to points in the plane, we swap their coordinates. 
 For a point $\alpha=(a,b)\in\Z^2$, we denote by $\n(\alpha):=b$ (resp., $\k(\alpha):=a$) its horizontal (resp., vertical) coordinate.
\end{notation}

\begin{figure}
\scalebox{1.0}{

\begin{tabular}{ccc}
\scalebox{0.60}{\begin{tikzpicture}[xscale=0.50, yscale=0.50, baseline=(ZUZU.base)]
	\coordinate(ZUZU) at (0,1.50);
\node[draw,circle,scale=0.30,fill=black] (T0) at (0,3.00) {};
\node[anchor=south,scale=1.00] (LT0) at (0,3.00) {$0$};
\node[draw,circle,scale=0.30,fill=black] (T1) at (1,3.00) {};
\node[anchor=south,scale=1.00] (LT1) at (1,3.00) {$1$};
\node[draw,circle,scale=0.30,fill=black] (T2) at (2,3.00) {};
\node[anchor=south,scale=1.00] (LT2) at (2,3.00) {$2$};
\node[draw,circle,scale=0.30,fill=black] (T3) at (3,3.00) {};
\node[anchor=south,scale=1.00] (LT3) at (3,3.00) {$3$};
\node[draw,circle,scale=0.30,fill=black] (T4) at (4,3.00) {};
\node[anchor=south,scale=1.00] (LT4) at (4,3.00) {$4$};
\node[draw,circle,scale=0.30,fill=black] (T5) at (5,3.00) {};
\node[anchor=south,scale=1.00] (LT5) at (5,3.00) {$5$};
\node[draw,circle,scale=0.30,fill=black] (T6) at (6,3.00) {};
\node[anchor=south,scale=1.00] (LT6) at (6,3.00) {$0$};
\node[draw,circle,scale=0.30,fill=black] (T7) at (7,3.00) {};
\node[anchor=south,scale=1.00] (LT7) at (7,3.00) {$1$};
\node[draw,circle,scale=0.30,fill=black] (T8) at (8,3.00) {};
\node[anchor=south,scale=1.00] (LT8) at (8,3.00) {$2$};
\node[draw,circle,scale=0.30,fill=black] (B3) at (3,0.00) {};
\node[anchor=north,scale=1.00] (LT3) at (3,0.00) {$3$};
\node[draw,circle,scale=0.30,fill=black] (B4) at (4,0.00) {};
\node[anchor=north,scale=1.00] (LT4) at (4,0.00) {$4$};
\node[draw,circle,scale=0.30,fill=black] (B5) at (5,0.00) {};
\node[anchor=north,scale=1.00] (LT5) at (5,0.00) {$5$};
\node[draw,circle,scale=0.30,fill=black] (B6) at (6,0.00) {};
\node[anchor=north,scale=1.00] (LT6) at (6,0.00) {$0$};
\node[draw,circle,scale=0.30,fill=black] (B7) at (7,0.00) {};
\node[anchor=north,scale=1.00] (LT7) at (7,0.00) {$1$};
\node[draw,circle,scale=0.30,fill=black] (B8) at (8,0.00) {};
\node[anchor=north,scale=1.00] (LT8) at (8,0.00) {$2$};
\node[draw,circle,scale=0.30,fill=black] (B9) at (9,0.00) {};
\node[anchor=north,scale=1.00] (LT9) at (9,0.00) {$3$};
\node[draw,circle,scale=0.30,fill=black] (B10) at (10,0.00) {};
\node[anchor=north,scale=1.00] (LT10) at (10,0.00) {$4$};
\node[draw,circle,scale=0.30,fill=black] (B11) at (11,0.00) {};
\node[anchor=north,scale=1.00] (LT11) at (11,0.00) {$5$};
\draw[black, line width=0.50pt] (T0)--(B3);
\draw[black, line width=0.50pt] (T1)--(B4);
\draw[black, line width=0.50pt] (T2)--(B5);
\draw[black, line width=0.50pt] (T3)--(B8);
\draw[black, line width=0.50pt] (T4)--(B6);
\draw[black, line width=0.50pt] (T5)--(B7);
\draw[black, line width=0.50pt] (T6)--(B9);
\draw[black, line width=0.50pt] (T7)--(B10);
\draw[black, line width=0.50pt] (T8)--(B11);
\node[scale=1.50, anchor=west] (DDDTR) at (8.30, 3.00) {$\dots$};
\node[scale=1.50, anchor=east] (DDDTR) at (-0.30, 3.00) {$\dots$};
\node[scale=1.50, anchor=west] (DDDTR) at (11.30, 0.00) {$\dots$};
\node[scale=1.50, anchor=east] (DDDTR) at (2.70, 0.00) {$\dots$};
\end{tikzpicture}
} & $\longrightarrow$ & \scalebox{0.80}{\begin{tikzpicture}[xscale=1.00, yscale=1.00,baseline=(ZUZU.base)]
\coordinate(ZUZU) at (0,1.50);
\draw[black!60,line width=0.3] (0,0) grid (6, 3);
\node[anchor=east,blue] (K) at (0.00,1.50) {$3$};
\node[anchor=south,blue] (N) at (3.00,3.00) {$6$};
\node[brown,scale=0.75,anchor=south east] (FL0) at (0, 0.000000000000000) {$P[0]$};
\node[brown,scale=0.1,draw,fill=brown,circle] (FD0) at (0, 0.000000000000000) {$0$};
\node[brown,scale=0.75,anchor=south east] (FL1) at (1, 0.500000000000000) {$P[3]$};
\node[brown,scale=0.1,draw,fill=brown,circle] (FD1) at (1, 0.500000000000000) {$3$};
\node[brown,scale=0.75,anchor=south east] (FL2) at (2, 1.33333333333333) {$P[2]$};
\node[brown,scale=0.1,draw,fill=brown,circle] (FD2) at (2, 1.33333333333333) {$2$};
\node[brown,scale=0.75,anchor=south east] (FL3) at (3, 1.83333333333333) {$P[5]$};
\node[brown,scale=0.1,draw,fill=brown,circle] (FD3) at (3, 1.83333333333333) {$5$};
\node[brown,scale=0.75,anchor=south east] (FL4) at (4, 2.16666666666667) {$P[1]$};
\node[brown,scale=0.1,draw,fill=brown,circle] (FD4) at (4, 2.16666666666667) {$1$};
\node[brown,scale=0.75,anchor=south east] (FL5) at (5, 2.66666666666667) {$P[4]$};
\node[brown,scale=0.1,draw,fill=brown,circle] (FD5) at (5, 2.66666666666667) {$4$};
\node[brown,scale=0.75,anchor=south east] (FL6) at (6, 3.00000000000000) {$P[n]$};
\node[brown,scale=0.1,draw,fill=brown,circle] (FD6) at (6, 3.00000000000000) {$6$};
\draw[brown, line width=0.5pt] (FD0) -- (FD1);
\draw[brown, line width=0.5pt] (FD1) -- (FD2);
\draw[brown, line width=0.5pt] (FD2) -- (FD3);
\draw[brown, line width=0.5pt] (FD3) -- (FD4);
\draw[brown, line width=0.5pt] (FD4) -- (FD5);
\draw[brown, line width=0.5pt] (FD5) -- (FD6);
\end{tikzpicture}
} \\ \vspace{-0.1in} \\
$\fb=(0,3,2,5,1,4)$ in cycle notation & & $\textcolor{brown}{P=\Pf}$
\end{tabular}

}
  \caption{\label{fig:path} Computing the small path $\Pf$. Its points are labeled according to  \cref{notn:P[*]}.}
\end{figure}
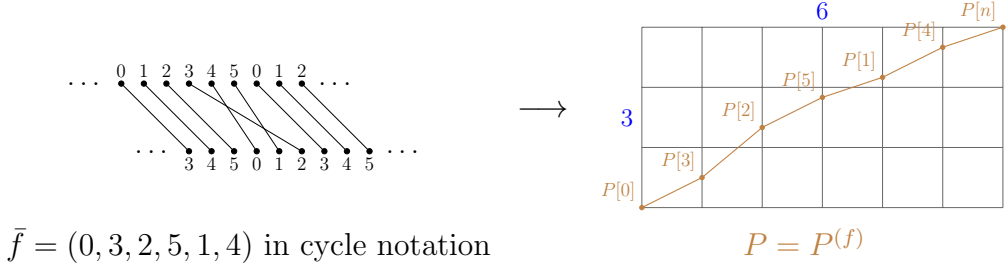

\begin{definition}
The \emph{big path} $\Pfinf$ of $f$ is the path in the plane through the points $\pf_r:=(f^r(0)/n,r)$ for all $r\in\Z$. The \emph{small path} $\Pf$ is the subpath of $\Pfinf$ through the points $\pf_0,\pf_1,\dots,\pf_n$.
\end{definition}
See \cref{fig:path}. We usually drop the superscript and denote $\Pinf:=\Pfinf$. We refer to the points $\pf_r$ for $r\in\Z$ as the \emph{integer points} of $\Pinf$.

Set $\del:=(k,n)$ and choose some $\alpha\in\Z^2$. We will be interested in the \emph{intersection points} of $\Pinf$ with $\Qinf:=\Pinf+\alpha$. First, observe that if $\alpha\in\Z\del$ then $\Pinf=\Qinf$. If $\alpha\notin\Z\del$ then it is easy to see that no integer point of $\Pinf$ belongs to $\Qinf$, and that the set $\Pinf\cap \Qinf$ is invariant under adding multiples of $\del$. We denote by $|\Pinf\cap\Qinf|$ the size of this set  when considered ``modulo $\del$,'' that is, as a subset of the cylinder\footnote{Some of our constructions are most naturally described in terms of the cylinder $\Z^2/\Z\del$. However, we choose to work with the full plane $\Z^2$ since we need to talk about convexity. For example, we will see that the set $\Fbf$  is convex as a subset of the plane but not as a subset of the cylinder.} $\Z^2/\Z\del$. For $l:=|\Pinf\cap\Qinf|$, we say that \emph{$\Pinf$ and $\Qinf$ intersect $l$ times}. The number $l$ is always finite and even. 

\begin{proposition}\label{prop:F_paths}
Let $f\in\Bknc$. Then $f$ is \repfree if and only if for all $\alpha\in\Z^2\setminus\Z\del$, $\Pinf$ and $\Qinf:=\Pinf+\alpha$ intersect at most two times. In this case, we have
\begin{equation*}%
  \F(f)=\{\alpha\in[k-1]\times[n-1]\mid \Pinf\text{ intersects }\Qinf\}.
\end{equation*}
\end{proposition}
\begin{proof}
Let $\alpha=(a,b)\neq(0,0)$. If $a\leq0, b\geq0$ or $a\geq0,b\leq0$ then clearly $\Pinf$ does not intersect $\Qinf$. Thus if $\Pinf$ intersects $\Qinf$ then there exists a unique $t\in\Z$ such that $\alpha+t\del\in[k-1]\times[n-1]$. From now on, we assume that $\alpha\in[k-1]\times[n-1]$.

We will prove the more general statement that for all $f\in\Bknc$, the multiplicity of $\alpha$ in the multiset $\F(f)$ is given by $\frac12|\Pinf\cap\Qinf|$. Indeed, suppose that $\Pinf$ crosses $\Qinf$ \emph{from below} at some non-integer point $x$. (That is, $\Pinf$ is below $\Qinf$ when approaching $x$ from the left and above $\Qinf$ when approaching $x$ from the right.)
 Then $x$ belongs to the segment of $\Pinf$ connecting $\pf_r$ to $\pf_{r+1}$ and to the segment of $\Qinf$ connecting $\alpha+\pf_{r-b}$ to $\alpha+\pf_{r-b+1}$, where $\alpha=(a,b)$. Let $i:=f^r(0)$ and $j:=an+f^{r-b}(0)$. Then we have $i<j$ and $f(i)>f(j)$, and thus $(i',j'):=(i-tn,j-tn)$ form an inversion of $f$, where $t\in\Z$ is such that $i-tn\in[n]$. Moreover, it is easy to see that $\kn(f^{(i',j')}_1)=\alpha$, where $\kn(\cdot)$ was defined in~\eqref{eq:del_dfn}.

Conversely, given an inversion $(i,j)$ of $f$ with $\kn(\fij_1)=\alpha$, we may find a (unique modulo $n$) index $r\in\Z$ such that $f^r(0)\equiv i$ modulo $n$, and we can also find a (unique modulo $\del$) shift $\alpha\in\Z^2$ such that $\Qinf$ passes through the point $\pf_r+(\frac{j-i}n,0)$. This shows that the inversions $(i,j)$ of $f$ satisfying $\kn(\fij_1)=\alpha$ are in bijection with the $\frac12|\Pinf\cap\Qinf|$ points where $\Pinf$ crosses $\Qinf$ from below.
\end{proof}

We say that a multiset $\Flet'$ is \emph{centrally symmetric} if for each $\alpha\in[k-1]\times[n-1]$, the multiplicities of $\alpha$ and of $\del-\alpha$ in $\Flet'$ coincide.
\begin{corollary}\label{cor:cs}
For all $f\in\Bknc$, the inversion multiset $\F(f)$ is centrally symmetric.
\end{corollary}
\begin{proof}
We showed above that the multiplicity of $\alpha$ in $\F(f)$ is given by $\frac12|\Pinf\cap(\Pinf+\alpha)|$. Since $\Pinf=\Pinf+\del$, we find $|\Pinf\cap(\Pinf+\del-\alpha)|=|\Pinf\cap(\Pinf+\alpha)|$, and the result follows.
\end{proof}

We discuss how $\F(f)$ changes under length-preserving simple conjugations and double moves. The following result is immediate.
\begin{lemma}
Let $f\in\Bknc$ be \repfree.
\begin{theoremlist}
\item If $g\ceq f$ then $g$ is \repfree and $\F(f)=\F(g)$.
\item If $f$ has a double crossing at $i\in[n]$ then $f':=s_ifs_i$ is \repfree and
\begin{equation*}%
  \F(f')=\F(f)\setminus \{\kn(\fii_1),\kn(\fii_2)\}.
\end{equation*}
\end{theoremlist}
\end{lemma}

For a point $\alpha=(a,b)\in[k-1]\times[n-1]$, let $\slope(\alpha):=\frac ab$. Part~\itemref{slope:=} of the next lemma confirms that $\F(f)$ always contains all points on the main diagonal of $[k-1]\times[n-1]$.
\begin{lemma}\label{lemma:slope}
Let $f\in\Bknc$, $\alpha\in[k-1]\times[n-1]$, and $\Qinf:=\Pinf+\alpha$. 
\begin{theoremlist}
\item\label{slope:<} If $\slope(\alpha)\leq \slope(\del)$ then $\Qinf$ contains integer points below $\Pinf$. 
\item\label{slope:>} If $\slope(\alpha)\geq \slope(\del)$ then $\Qinf$ contains integer points above $\Pinf$.
\item\label{slope:=} If $\slope(\alpha) = \slope(\del)$ then $\alpha$ belongs to $\F(f)$.
\end{theoremlist}
\end{lemma}
\begin{proof}
For $r\in\Z$, let $\qf_r$ be the integer point of $\Qinf$ with horizontal coordinate $\n(\qf_r)=r$. 
Let $\delp:=(1,-k/n)$ and denote by $\<\cdot,\cdot\>$ the standard dot product on $\R^2$. We have $\<\delp,\del\>=0$, and the sign of $\<\delp,\alpha\>$ coincides with the sign of $\slope(\alpha)-\slope(\del)$. Let
\begin{equation}\label{eq:delp_Pinf}
  \<\delp,\Pinf\>:=\sum_{r=0}^{n-1} \<\delp,\pf_r\> \quad\text{and}\quad   \<\delp,\Qinf\>:=\sum_{r=0}^{n-1} \<\delp,\qf_r\>.
\end{equation}
Since $\pf_{r+n}=\pf_r+\del$ for all $r\in\Z$, we have $\<\delp,\Pinf\>=\sum_{r=j}^{j+n-1}\<\delp,\pf_r\>$ for all $j\in\Z$, and similarly for $\<\delp,\Qinf\>$. In particular, we have 
\begin{equation*}%
  \<\delp,\Qinf\>=\sum_{r=0}^{n-1} \<\delp,\pf_r+\alpha\>=\<\delp, \Pinf\>+n\<\delp,\alpha\>.
\end{equation*}
Observe that for each $r\in\Z$, $\qf_r$ is above $\Pinf$ if and only if it is above $\pf_r$, which happens if and only if $\<\delp,\qf_r-\pf_r\>>0$, since the vertical coordinate of $\delp$ is positive.  Thus~\itemref{slope:<}--\itemref{slope:>} follow, and~\itemref{slope:=} follows by combining~\itemref{slope:<}--\itemref{slope:>} with (the proof of) \cref{prop:F_paths}, since if $\slope(\alpha)=\slope(\del)$ then $\Qinf$ contains integer points both below and above $\Pinf$, and therefore intersects $\Pinf$.
\end{proof}

\section{Convexity of the inversion multiset}\label{sec:convex}

Similarly to \cref{sec:main-result}, we say that $\F(f)$ is \emph{convex} if $\Fbf$ contains all lattice points of its convex hull. (These sets were defined in the beginning of \cref{sec:big-paths}.) The goal of this section is to prove the following result. %

\begin{theorem}\label{thm:convex}
Let $f\in\Bknc$ be \repfree. Then the set $\F(f)$ is convex.
\end{theorem}

We start by stating some consequences of the results obtained in \cref{sec:affine}. Let
\begin{equation*}%
  \Fmin=\{\alpha\in[k-1]\times[n-1]\mid \slope(\alpha)=\slope(\del)\}.
\end{equation*}
 By \cref{slope:=}, we have $\Fmin\subseteq \F(f)$ for all $f\in\Bknc$. The next two statements follow directly from \cref{prop:structure}.
\begin{corollary}\label{find_double_cross}
Let $f\in\Bknc$ be \repfree. Then at least one of the following holds:
\begin{itemize}
\item $\F(f)=\Fmin$.
\item There exists $g\in\Bknc$ such that $f\ceq g$ and $g$ has a double crossing at some $i\in\Z$.
\end{itemize}
\end{corollary}

\begin{corollary}\label{Fmin_MSeq}
Suppose that $f,g\in\Bknc$ are \repfree and $\F(f)=\F(g)=\Fmin$. Then $\Cat_f=\Cat_g$. 
\end{corollary}

Throughout the rest of this section, the following data is fixed:
\begin{itemize}
\item a \repfree $f\in\Bknc$ that has a double crossing at $0$;
\item a big path $\Pinf:=\Pfinf$ and a small path $P:=\Pf$ for $f$;
\item $f_1:=f^{(0,1)}_1\in \Bxc_{k_1,n_1}$ and $f_2:=f^{(0,1)}_2\in\Bxc_{k_2,n_2}$ obtained by resolving the crossing $(0,1)$ as in \cref{sec:intro:rep-free};
\item $\del:=(k,n)$, $\del_1:=(k_1,n_1)$, and $\del_2:=(k_2,n_2)$.
\end{itemize}

\begin{notation}\label{notn:P[*]}
For each $0\leq r<n$, let $0\leq \lab_r<n$ be the unique index equal to $f^r(0)$ modulo $n$. Then we label $\pf_r$ by $P[\lab_r]$ as in \cref{fig:path}. We extend this to all $r\in\Z$ using the convention that $j_{r+n}:=j_r+n$, and we label $\pf_r$ by $P\subfty[\lab_r]$ for $r\in\Z$. Thus $P\subfty[j+n]=P\subfty[j]+\del$ for all $j\in\Z$. If $P\subfty[i]$ appears to the left of $P\subfty[j]$ for some $i,j\in\Z$, we denote by $P\subfty[i\to j]$ the subpath of $\Pinf$ connecting $P[i]$ to $P[j]$. Thus $P=P\subfty[0\ton]$ and we will be particularly interested in the subpaths $P[0\to1]$ and $P[1\ton]$ of $P$ (under the above assumption that $f$ has a double crossing at $0$).
\end{notation}

We establish several elementary properties of $\F(f)$. For a subset $\Flet'\subset\Z^2$, we let $\Flet'+\Z\del:=\{ \alpha+t\del\mid \alpha\in \Flet', t\in\Z\}$. Recall also that we set $\delp:=(1,-k/n)$ and that the sign of $\<\delp,\alpha\>$ is positive if and only if $\alpha$ is above the line spanned by $\del$. Finally, we adopt the convention that when we shift a (big or small) path, its labeling of points from \cref{notn:P[*]} is preserved; for example, $(P+\alpha)[1\ton]:=P[1\ton]+\alpha$, etc.
\begin{lemma}\ \label{lemma:F_properties_easy}
\begin{theoremlist}
\item\label{item:semigroup} Let $\alpha,\beta\in \Z^2$ be such that $\<\delp,\alpha\><0$, $\<\delp,\beta\><0$, and $\alpha,\beta\notin (\Fbf+\Z\del)$. Then $\alpha+\beta\notin (\Fbf+\Z\del)$.
\item\label{item:sl_1<sl<sl_2} We have $\slp(\del_1)<\slp(\del)<\slp(\del_2)$.
\item\label{item:order_ideal} Let $\alpha\in\F(f)$ be such that $\<\delp,\alpha\><0$. Then for all $\beta\in\Z^2$ satisfying $\<\delp,\beta\>\leq0$, $\n(\beta)\leq \n(\alpha)$, and $\k(\beta)\geq\k(\alpha)$, we have $\beta\in \F(f)$.
\end{theoremlist}
\end{lemma}
\begin{proof}
\itemref{item:semigroup}: We showed in the proof of \cref{lemma:slope} that if $\<\delp,\alpha\><0$ then $\Pinf+\alpha$ contains integer points below $\Pinf$. If in addition $\alpha\notin (\Fbf+\Z\del)$ then $\Pinf+\alpha$ and $\Pinf$ do not intersect. Thus $\Pinf+\alpha$ and $\Pinf+\beta$ are both below $\Pinf$. But then $\Pinf+\alpha+\beta$ is below $\Pinf+\alpha$, and therefore it is below $\Pinf$, so $\alpha+\beta\notin (\Fbf+\Z\del)$.

\itemref{item:sl_1<sl<sl_2}: The only integer points of $\Pinf+\del_1$ above $\Pinf$ are $P[1]+\Z\del$. Thus $\Pinf+2\del_1$ is below $\Pinf$ (cf. \cref{rmk:semigroup_strong} below). By \cref{slope:>}, we must have $\slp(\del_1)=\slp(2\del_1)<\slp(\del)$. Similarly, $\slp(\del)<\slp(\del_2)$.

\itemref{item:order_ideal}: If $\beta=\alpha$ then clearly $\beta\in\F(f)$. Assume that $\beta\neq\alpha$ and let $e:=\alpha-\beta$. We have $e\neq0$, $\n(e)\geq0$, and $\k(e)\leq0$, so $e\notin \Fbf+\Z\del$ and $\<\delp,e\><0$. If $\beta\notin(\Fbf+\Z\del)$ then by~\itemref{item:semigroup}, we must have $\alpha\notin(\Fbf+\Z\del)$, a contradiction. Thus $\beta\in(\Fbf+\Z\del)$, and the conditions on the coordinates of $\beta$ ensure that in fact $\beta\in\F(f)$.
\end{proof}

\begin{remark}\label{rmk:semigroup_strong}
  The proof of~\itemref{item:semigroup}--\itemref{item:sl_1<sl<sl_2} above shows the stronger statement that if $\<\delp,\alpha\><0$, $\<\delp,\beta\><0$, and $\alpha,\beta\notin ((\Fbf\setminus\{\del_1\})+\Z\del)$ then $\alpha+\beta\notin (\Fbf+\Z\del)$.
\end{remark}

\begin{figure}
  \includegraphics{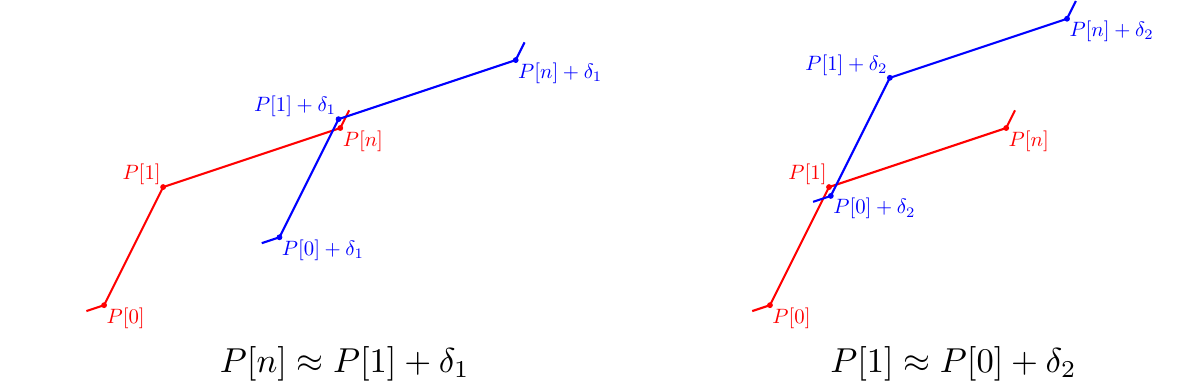}
  \caption{\label{fig:P_double} We write $P[n]\approx P[1]+\del_1$ and $P[1]\approx P[0]+\del_2$; see \cref{notn:approx}.}
\end{figure}

\begin{notation}\label{notn:approx}
Observe that the points $P[1]$ and $P[0]+\del_2$ differ by $(1/n,0)$. Moreover, the two paths $\Pinf$ and $\Pinf+\del_2$ form a double crossing at these two points, thus they form a small region  as in \cref{fig:P_double}. Therefore no shift of $\Pinf$ can contain an integer point in this region. In our analysis, we usually treat this region as a ``single point'' and write $P[1]\approx P[0]+\del_2$ and $P[n]\approx P[1]+\del_1$. By an abuse of terminology, we will say that $\Pinf$ is \emph{below} $\Pinf+\del_2$ and \emph{above} $\Pinf+\del_1$.
\end{notation}

\begin{lemma}
The bounded affine permutations $f_1$ and $f_2$ are \repfree.
\end{lemma}
\begin{proof}
Let us compare the big path $\Pinf^{(f_1)}$ with $(P[1\ton])_\infty:=\bigcup_{t\in\Z}(P[1\ton]+t\del_1)$, where we identify the points $P[n]+(t-1)\del_1\approx P[1]+t\del_1$ for all $t\in\Z$. It is easy to see that these two paths are \emph{equivalent} in the sense that for each $\alpha=(a,b)\in\Z^2$, we have 
\begin{equation*}%
  |\Pinf^{(f_1)}\cap (\Pinf^{(f_1)}+\alpha)|=|(P[1\ton])_\infty\cap ((P[1\ton])_\infty+\alpha)|,
\end{equation*} 
where the intersection points are counted modulo $\del_1$. Thus we need to analyze the intersections of $(P[1\ton])_\infty$ with its shifts. 

Let $\alpha\in[k_1-1]\times[n_1-1]$ and $Q:=P+\alpha$. Let $s:=|(P[1\ton])_\infty\cap (Q[1\ton])_\infty|$. Since $f$ is \repfree, $P[1\ton]$ intersects $Q[1\ton]$ at most twice. Moreover, $P[1\ton]$ can intersect $Q[1\ton]+t\del_1$ only for $t\in\{-1,0\}$. Thus $s\leq 4$.

Suppose that $s>2$. Since $s$ is even, we have $s\geq 4$, and thus $s=4$. We see that $P[1\ton]$ intersects each of $Q[1\ton]$ and $Q':=Q[1\ton]-\del_1$ twice. Suppose first that $Q[1]$ is above $\Pinf$. Then $Q[1-n\to1]$ stays below $Q'[1\ton]$ which intersects $P[1\ton]$, and therefore $Q[1-n\to1]$ intersects $P[1\ton]$. We have found three intersection points of $P[1\ton]$ with $\Qinf$, a contradiction. Suppose now that $Q[1]$ is below $\Pinf$. Then $Q'\subfty[n\to 2n]$ stays above $Q[1\ton]$ which intersects $P[1\ton]$, and thus $Q'\subfty[n\to2n]$ intersects $P[1\ton]$.  We have found three intersection points of $P[1\ton]$ with $Q'_\infty$, a contradiction. We have shown that $f_1$ is \repfree.

Similarly, we check that $\Pinf^{(f_2)}$ is equivalent to $(P[0\to1])_\infty:=\bigcup_{t\in\Z}(P[0\to1]+t\del_2)$ and use it to deduce that $f_2$ is \repfree.
\end{proof}

Note that $P[1\ton]$ has \emph{low slope} since it connects $P[1]$ to $P[n]\approx P[1]+\del_1$  while $P[0\to1]$ has \emph{high slope} since it connects $P[0]$ to $P[1]\approx P[0]+\del_2$; cf. Lemma~\directref{lemma:F_properties_easy}{item:sl_1<sl<sl_2}.%
 The next result states that a shifted segment of high slope cannot cross a segment of low slope from above.

\begin{proposition}\label{prop:cross_below}
Let $\alpha\in\Z^2$ and $Q:=P+\alpha$. Then $Q[0 \to 1]$ cannot cross $P[1\ton]$ from above.
\end{proposition}
\begin{proof}
Suppose otherwise that $Q[0 \to 1]$ crosses $P[1\ton]$ from above. We consider the cases according to the positions of $Q[0]$ and $Q[1]$ relative to $\Pinf$.  First, assume that $Q[0]$ is below $\Pinf$ and $Q[1]$ is above $\Pinf$. Then $Q[0\to1]$ intersects $\Pinf$ at least $3$ times, a contradiction.

From now on we assume that $Q[0]$ is above $\Pinf$. (The case of $Q[1]$ being below $\Pinf$ is completely analogous.) 
Let $P':=P+\del_2$ and $Q':=Q+\del_2$. Since $Q[0]$ is above $\Pinf$, $Q[1]\approx Q'[0]$ are both above $\Pinf'$. Moreover, $Q$ crosses $P[1\ton]$ from above, thus $Q'$ crosses $P'[1\ton]$ from above.

\begin{definition}\label{dfn:vert_above}
For an integer point $q$ of $Q$, we say that \emph{$q$ is vertically above $P$} if there exists an integer point $p$ of $P$ with $\n(q)=\n(p)$, and $q$ is above $p$. 
\end{definition}

\begin{figure}
\includegraphics[width=0.4\textwidth]{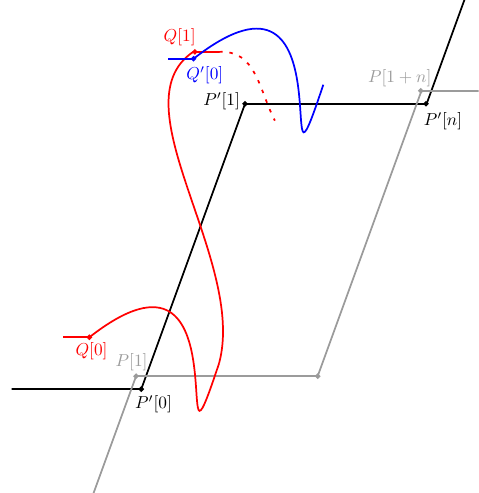}
  \caption{\label{fig:cross_below} Proof of \cref{prop:cross_below}.}
\end{figure}
 
Since $Q'[0\to1]$ intersects $P'[1\ton]$, it follows that $Q'[0]$ is vertically above $P'$. Consider the path $Q\subfty[1\to1+n]$. It crosses $\Pinf$ from above at a single point which belongs to $Q\subfty[n\to1+n]\cap P\subfty[1+n\to2n]$. Moreover, it stays below $Q'$ which crosses $P'[1\ton]$ from above. Thus $Q\subfty[1\to1+n]$ crosses $P'$ from above. Since $Q\subfty[1\to1+n]$ cannot cross $P\subfty[1\to1+n]$ from above, it must cross $P'$ from below. The remaining part of $Q\subfty[1\to1+n]$ still has to cross $\Pinf$ from above, however, it cannot cross $\Pinf'$ since it has already crossed $P'$ twice. Since $\Pinf$ is below $\Pinf'$, we get a contradiction. See \cref{fig:cross_below}. 
\end{proof}

\begin{lemma}\label{lemma:inf_slopes}
Let $\alpha,\beta\in [k-1]\times [n-1]$ be such that $\slp(\alpha)>\slp(\beta)$. Suppose that there are two subpaths $P[a\tob]$ and $P[c\tod]$ of $\Pinf$ such that $P[a\tob]$ crosses $P[a\tob]+\alpha$ while $P[c\tod]$ crosses $P[c\tod]+\beta$. Then there exist $s,t\in\Z$ such that $P[a\tob]+s\alpha$ crosses $P[c\tod]+t\beta$ from below.
\end{lemma}
\begin{proof}
% We may assume that the subpaths $P[a\tob]$ and $P[c\tod]$ are minimal by inclusion such that $P[a\tob]$ intersects $P[a\tob]+\alpha$ while $P[c\tod]$ intersects $P[c\tod]+\beta$. Thus the last segment $P[b\to b']$ of $P[a\tob]$ intersects $P[a\tob]+\alpha$, the first segment $P[a\to a']$ of $P[a\tob]$ intersects $P[a\tob]-\alpha$, and similarly for $P[c\to c']$ and $P[d'\to d]$. We refer to $P[a\to a']$, $P[b\to b']$, $P[c\to c']$, and $P[d'\to d]$ as the \emph{boundary} segments of $P[a\tob]$ and $P[c\tod]$.

Consider the two infinite unions $\Ra:=P[a\tob]+\Z\alpha$ and $\Rb:=P[c\tod]+\Z\beta$. Observe that $\Ra$ (resp., $\Rb$) is a path-connected subset of $\R^2$. Thus it contains an infinite piecewise linear curve $S_\alpha$ (resp., $S_\beta$) such that for each $r\in\Z$, $S_\alpha$ (resp., $S_\beta$) contains a unique point $x_{\alpha,r}$ (resp., $x_{\beta,r}$) satisfying $\n(x_{\alpha,r})=\n(x_{\beta,r})=r$. Here, we are additionally assuming that the vertical coordinates of $x_{\alpha,r}$ and $x_{\beta,r}$ are increasing functions of $r$. 

When $r\ll 0$, $x_{\alpha,r}$ is below $x_{\beta,r}$, and when $r\gg0$, $x_{\alpha,r}$ is above $x_{\beta,r}$. Let $r\in\Z$ be the smallest integer such that $x_{\alpha,r}$ is \emph{not} below $x_{\beta,r}$. Thus $x_{\alpha,r-1}$ is below $x_{\beta,r-1}$ and either $x_{\alpha,r}=x_{\beta,r}$ or $x_{\alpha,r}$ is above $x_{\beta,r}$. In each case, it is straightforward to check that a shift $P[a\tob]+s\alpha$ (passing through either $x_{\alpha,r}$ or $x_{\alpha,r-1}$ or both) crosses a shift $P[c\tod]+t\beta$ (passing through either $x_{\beta,r}$ or $x_{\beta,r-1}$ or both) from below.
\end{proof}
\begin{remark}\label{rmk:inf_slopes}
The same argument applies when either $(a,b,\alpha)=(1,n,\del_1)$ or $(c,d,\beta)=(0,1,\del_2)$. (Since $\slp(\del_1)<\slp(\del_2)$ by Lemma~\directref{lemma:F_properties_easy}{item:sl_1<sl<sl_2}, we cannot have both.) Suppose for instance that $(c,d,\beta)=(0,1,\del_2)$. Even though $P[0\to1]$ does not intersect $P[0\to1]+\del_2$, since we identify $P[1]\approx P[0]+\del_2$, the union $\Rb=P[0\to1]+\Z\del_2$ still contains an infinite connected (modulo our identification) piecewise linear curve.
\end{remark}

Given two paths $Q,P$, we say that \emph{$Q$ is above $P$} if whenever two integer points $q\in Q$ and $p\in P$ satisfy $\n(q)=\n(p)$, we have that $q$ is above $p$. (This condition is vacuously true if the projections of $Q$ and $P$ onto the horizontal axis do not overlap.)

\begin{lemma}\label{lemma:impossible}
Let $\alpha\in[k-1]\times[n-1]$ and $Q:=P+\alpha$. Assume that $Q[1]$ is above $\Pinf$. Then $Q[1\ton]$ and $P[1\ton]$ cannot intersect twice. 
\end{lemma}
\begin{proof}
Assume otherwise that they intersect twice. Our temporary goal is to show that 
\begin{equation}\label{eq:imposs:Pinf_below_Qinf}
  \text{$(P[0\to1])_\infty$ is below $(Q[0\to1])_\infty$.}
\end{equation}
We observe that $Q$ satisfies the following properties:
\begin{theoremlist}[label=(\alph*)]
\item\label{item:imposs:1} $Q[1]$ is above $\Pinf$;
\item\label{item:imposs:2} $Q[1\ton]$ intersects $P[1\ton]$ twice;
\item\label{item:imposs:3} $\n(Q[1])\geq \n(P[1])$.
\end{theoremlist}
Let $Q':=Q-\del_2$ so that $Q'[1]\approx Q[0]$. In order to show~\eqref{eq:imposs:Pinf_below_Qinf}, it suffices to prove that if $Q$ satisfies~\itemref{item:imposs:1}--\itemref{item:imposs:3} then either 
\begin{theoremlist}
\item\label{item:imp_concl1} $Q'$ satisfies~\itemref{item:imposs:1}--\itemref{item:imposs:3}, or
\item\label{item:imp_concl2} $Q'[0\to1]$ is above $P[0\to1]$ with $\n(Q'[0])<\n(P[0])$.
\end{theoremlist}
If~\itemref{item:imp_concl1} holds for $Q'$ then we proceed by induction, applying the same argument to $Q'-t\del_2$ for $t=1,2,\dots$, until we find that~\itemref{item:imp_concl2} holds for some $Q-s\del_2$ with $s>0$. But then all integer points of $P[0\to1]$ are below $\bigcup_{t=0}^s(Q-t\del_2)$, which proves~\eqref{eq:imposs:Pinf_below_Qinf}.

Assume that $Q$ satisfies~\itemref{item:imposs:1}--\itemref{item:imposs:3}. Since $Q[1\ton]$ and $P[1\ton]$ intersect twice, $Q[0\to1]$ is above $\Pinf$ and $P[0\to1]$ is below $\Qinf$. Then $Q'[0\to1]$ is  above $\Pinf'$, where $P':=P-\del_2$. We have the following situation:
\begin{itemize}
\item apart from the double crossing, $Q\subfty'[1-n\to1]$ is below $\Qinf$;
\item $Q\subfty'[1-n]$ and $Q'[1]$ are above $\Pinf$;
\item $Q\subfty'[1-n\to0]$ intersects $P\subfty'[1-n\to0]$ twice;
\item $P'[1]\approx P[0]$ is below $\Qinf'$.
\end{itemize}
These statements imply that  $Q\subfty'[1-n\to1]$ crosses $P\subfty[-n\to0]$ twice, first from above and then from below. Moreover, the second crossing (from below) must belong to $P\subfty[1-n\to0]$ since it has to come after both crossings of $Q\subfty[1-n\to0]$ with $P\subfty'[1-n\to0]$. In particular, no part of $Q\subfty'[1-n\to1]$ is below $P[0\to1]$, and thus $Q'[0\to1]$ is above $P[0\to1]$.

Suppose that $\n(Q'[1])<\n(P[1])$. Then $\n(Q'[0])<\n(P[0])$. We have just shown that $Q'[0\to1]$ is above $P[0\to1]$, so we arrive at case~\itemref{item:imp_concl2}.

Suppose now that $\n(Q'[1])\geq \n(P[1])$. Then $Q'$ satisfies~\itemref{item:imposs:1} and~\itemref{item:imposs:3}. Moreover, we also have $\n(Q\subfty'[1-n])\geq\n(P\subfty[1-n])$ and $\n(Q[0])\geq\n(P[0])$. In view of the above statements, this implies that $Q\subfty'[1-n\to0]$ intersects $P\subfty[1-n\to0]$ twice, i.e., $Q'$ also satisfies~\itemref{item:imposs:2}. We arrive at case~\itemref{item:imp_concl1}. We are done with the proof of~\eqref{eq:imposs:Pinf_below_Qinf}, and now we will use it to finish off the proof of the lemma. 

Observe that if $\slope(\alpha)\leq \slope(\del_2)$ then we get a contradiction by \cref{slope:<} and~\eqref{eq:imposs:Pinf_below_Qinf}. Thus $\slope(\alpha)>\slope(\del_2)$. By \cref{lemma:inf_slopes} and \cref{rmk:inf_slopes}, for some $s,t, \in \Z$, we have that $P[1\ton]+s\alpha$ crosses $P[0\to1]+t\del_2$ from below, contradicting \cref{prop:cross_below}. 
\end{proof}

\begin{lemma}\label{lemma:corner_slope_weak}
Let $\alpha\in \F(f)$. 
\begin{theoremlist}
\item\label{item:corner_slope_weak1} If $\slope(\alpha)\leq \slope(\del_1)$ then $\alpha\in \Fb(f_1)+\Z\del_1$.
\item\label{item:corner_slope_weak2} If $\slope(\alpha)\geq \slope(\del_2)$ then $\alpha\in \Fb(f_2)+\Z\del_2$.
\end{theoremlist}
\end{lemma}
\begin{proof} 
We prove~\itemref{item:corner_slope_weak1}.  The proof of~\itemref{item:corner_slope_weak2} is completely analogous.

First, if $\slope(\alpha)=\slope(\del_1)$ then $\alpha\in \Fb(f_1)+\Z\del_1$ by \cref{slope:=}. Assume that $\slope(\alpha)<\slope(\del_1)$. By \cref{lemma:inf_slopes} and \cref{rmk:inf_slopes}, there are $s,t\in\Z$ such that  $P[1\ton]+s\del_1$ crosses $P+t\alpha$ from below. This crossing cannot belong to $P[0\to1]+t\alpha$ by \cref{prop:cross_below}. Thus it belongs to $P[1\ton]+t\alpha$, so $t\alpha\in \Fb(f_1)+\Z\del_1$. By Lemma~\directref{lemma:F_properties_easy}{item:semigroup}, we get $\alpha\in \Fb(f_1)+\Z\del_1$.
\end{proof}

The following straightforward result describes a natural transformation that swaps the notions of ``above'' and ``below.'' We refer to it as the \emph{\rotn-rotation}.

\begin{proposition}%
\label{prop:rotn}
For $f\in\Bknc$, let $g:\Z\to\Z$ be given by 
\begin{equation*}%
  g(j):=n-f^{-1}(-j) \quad\text{for all $j\in\Z$.}
\end{equation*}
Then $g\in\Bknc$ and the paths $\Pf$ and $\Pg$ are related as
\begin{equation}
  \Pg=\del-\Pf.
\end{equation}
For each point $x\in\R^2$, $x$ is above (resp., below) $\Pfinf$ if and only if $\del-x$ is below (resp., above) $\Pginf$. \qed
\end{proposition}

Given $\alpha,\beta\in\R^2$, we say that $\alpha$ is \emph{weakly southwest} of $\beta$ and write $\alpha\wsw \beta$ if $\n(\alpha)\leq \n(\beta)$ and $\k(\alpha)\leq \k(\beta)$. We write $\alpha\sw\beta$ if $\alpha\wsw\beta$ and $\alpha\neq\beta$.

\begin{lemma}\label{lemma:southwest}
Let $\alpha\in \F(f)$ and $\beta:=\del-\alpha$.
\begin{theoremlist}
\item\label{item:southwest1} If $\slope(\alpha)\leq \slope(\del_1)$ and $\alpha\sw\del_1$ then $\beta\notin \Fb(f_2)+\Z\del_2$.
\item\label{item:southwest2} If $\slope(\beta)\geq \slope(\del_2)$ and $\beta\sw\del_2$ then $\alpha\notin \Fb(f_1)+\Z\del_1$.
\end{theoremlist}
\end{lemma}
\begin{proof}
In view of \cref{prop:rotn}, we only prove~\itemref{item:southwest1}. By \cref{cor:cs}, we have $\beta\in\F(f)$. Let $Q:=P+\alpha$. By \cref{lemma:corner_slope_weak}, $\alpha\in \Fb(f_1)+\Z\del_1$ and $\beta\in \Fb(f_2)+\Z\del_2$, or equivalently, $Q[1\ton]$ intersects $(P[1\ton])_\infty$ twice and $Q[0\to1]$ intersects $L:=(P[n\to1+n])_\infty$ twice. Moreover, since $(0,0)\sw\alpha\sw\del_1$, we see that 
\begin{equation}\label{eq:wsw}
  P[0]\sw Q[0]\sw Q[1]\sw P[n].%
\end{equation}

\begin{figure}
\includegraphics[width=0.35\textwidth]{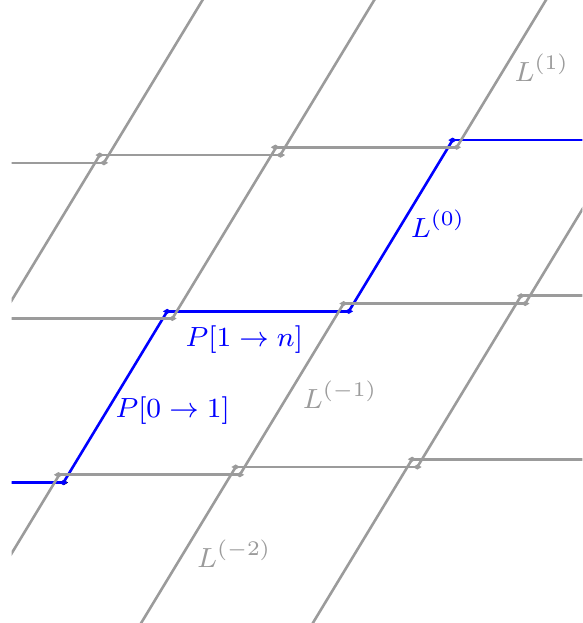}
  \caption{\label{fig:net} The proof of \cref{lemma:southwest}.}
\end{figure}

For $t\in\Z$, we let $\PX t:=\Pinf+t\del_2=\Pinf-t\del_1$ and $\LX t:=L\cap \PX t$; see \cref{fig:net}. Thus there exists a unique integer $t\in\Z$ such that $Q[1]$ is above $\PX t$ and below $\PX{t+1}$. We know that $Q[0\to1]$ intersects $L$, and by~\eqref{eq:wsw}, it can only intersect $\LX{<0}:=\bigcup_{s<0} \LX s$. Similarly, we observe that $Q[1\ton]$ must intersect $P[1\ton]\cup R[1\ton]$ exactly twice, where $R:=P+\del_1$ is a subpath of $\PX{-1}$. We consider four cases.

\cas1 $t\geq0$ and $Q[0]$ is above $\Pinf$. Since $\LX{<0}$ is below $\Pinf$ and $Q[0\to1]$ intersects it twice, we see that $Q[0\to1]$ intersects $\Pinf$ twice. Then $Q[1\ton]$ cannot intersect $\Pinf$, so it has to stay above $\Pinf$. Thus $Q[1\ton]$ must intersect $R[1\ton]$ twice, which is impossible since $R[1\ton]$ is below $\Pinf$. 

\cas2 $t\geq0$ and $Q[0]$ is below $\Pinf$. Thus $Q[0\to1]$ intersects $\Pinf$ once, and therefore so does $Q[1\ton]$. Thus $Q[1\ton]$ must also intersect $R[1\ton]$. Since $Q[1]$ is above $\Pinf$, $Q[n]\approx Q[1]+\del_1$ is above $\PX{-1}=\Pinf+\del_1$. Since $Q[1]$ is also above $\PX{-1}$, we see that $Q[1\ton]$ intersects $\PX{-1}$ twice. Since $Q[n]$ is above $\PX{-1}$, so is $Q[0]$. In order for $Q[0\to1]$ to intersect $\LX{<0}$, it must intersect $\PX{-1}$, since each point of $\LX{<0}$ is either on or below $\PX{-1}$. Thus $Q$ intersects $\PX{-1}$ at least three times, a contradiction.

\cas3 $t<0$ and $Q[1]$ is above $L$. Then $Q[0]$ is above $L$. Recall that $Q[1]$ is above $\PX t$ and below $\PX{t+1}$, thus $Q[0]$ and $Q[n]$ are above $\PX{t-1}$ and below $\PX t$. We see that each of $Q[0\to1]$ and $Q[1\ton]$ intersects $\PX t$ exactly once. Let $q$ be the first intersection point of $Q[0\to1]$ with $\PX t\cup\PX{t-1}$. We claim that $q$ belongs to $\PX{t-1}$. Indeed, suppose otherwise that $q\in \PX t$. Since $Q[0\to1]$ has to intersect $L$ but it can no longer intersect $\PX{t}$, the first intersection point $\ell$ of $Q[0\to1]$ with $L$ has to belong to $\LX s$ for some $s>t$. In order for this to happen, $Q[0\to1]$ must intersect $\PX s$ twice (with the second crossing at $\ell$), and therefore the remaining part of $Q[0\to 1]$ will stay below $\PX s$, and thus below $L$. We see that $Q[1]$ is below $L$, contradicting our assumption. Thus $q\in \PX{t-1}$. Specifically, we have $q\in (P[1\to1+n]+(t-1)\del_2)$, which is the lower boundary of the region bounded by $\PX{t-1}\cup\PX t$ containing $Q[0]$.

Consider $Q':=Q+\del_2$. The first intersection point $q'=q+\del_2$ of $Q'$ with $\PX t\cup\PX{t+1}$ belongs to $\PX{t}$. Since $Q[1\ton]$ stays below $Q'$ and intersects $\PX t$, we see that $Q[1\ton]$ stays below $\PX{t+1}$, and the unique intersection point of $Q[1\ton]$ with $\PX t$ belongs to $P[1\to1+n]+t\del_2$. Since $t<0$, $P[1\ton]$ stays above $\PX t$, and thus $Q[1\ton]$ must intersect $R[1\ton]$. If $t<-1$ then $R[1\ton]$ is above $\PX t$ and we get a contradiction. If $t=-1$ then we have already shown that the only intersection point of $Q[1\ton]$ with $\PX{-1}$ belongs to $P[1\to1+n]+t\del_2$ which is disjoint from $R[1\ton]$. 

\cas4 $t<0$ and $Q[1]$ is below $L$. Then $Q[0]$ is below $L$. It is still true that each of $Q[0\to1]$ and $Q[1\ton]$ intersects $\PX t$ exactly once. Thus the second point of $Q[0\to1]\cap L$ belongs to $\LX s$ for some $s>t$. Thus $Q[0\to1]$ intersects $\PX s$ twice, so $Q[1\ton]$ stays below $\PX s$. Since $Q[0\to1]$ can only intersect $\LX{<0}$, we find that $s<0$. Thus $Q[1\ton]$ cannot intersect $P[1\ton]\cup R[1\ton]$, a contradiction.
\end{proof}

Using \cref{lemma:southwest}, the result of \cref{lemma:corner_slope_weak} can be strengthened as follows.
\begin{corollary}\label{cor:corner_slope_stong}
Let $\alpha\in \F(f)$. 
\begin{theoremlist}
\item\label{item:corner_slope_strong1} If $\slope(\alpha)\leq \slope(\del_1)$ then $\alpha\in \F(f_1)\sqcup\{\del_1\}$.
\item\label{item:corner_slope_strong2} If $\slope(\alpha)\geq \slope(\del_2)$ then $\alpha\in \F(f_2)\sqcup\{\del_2\}$.
\end{theoremlist}
\end{corollary}
\begin{proof}
Again, by \cref{prop:rotn}, it suffices to prove~\itemref{item:corner_slope_strong1}. 
 By \cref{item:corner_slope_weak1}, we have $\alpha\in  \Fb(f_1)+\Z\del_1$, and recall that $\alpha\in[k-1]\times[n-1]$ since $\alpha\in \F(f)$. Let $m\in\Z$ be the unique integer satisfying $\alpha\in \Fb(f_1)+m\del_1$. Clearly, $m\geq0$. Our goal is to show that $m=0$. Assume for the sake of contradiction that $m>0$. Thus $\del_1\sw \alpha$. Let $\beta:=\del-\alpha$, then $\beta\sw\del_2$ and $\slp(\beta)\geq\slp(\del_2)$. We get a contradiction by \cref{item:southwest2}.
\end{proof}

\begin{lemma}\label{lemma:del_1-alpha}
Let $\alpha\in \F(f)\setminus \{\del_1\}$ be such that $\slope(\alpha)\leq \slope(\del_1)$. Then $\del_1-\alpha\in \F(f)$.
\end{lemma}
\begin{proof}
Let $x:=\del_1-\alpha$, and assume $x\notin \F(f)$. By \cref{item:corner_slope_strong1}, we have $\alpha\in \F(f_1)$, so there exists $s\in\Z$ such that for $Q:=P+\alpha-s\del_1$, we have that $Q[1\ton]$ crosses $P[1\ton]$. Since $x\notin \F(f)$, we cannot have $s=1$. If $s\notin \{0,1\}$ then $Q[1\ton]$ and $P[1\ton]$ cannot intersect at all because $\alpha\in \F(f_1)$ implies $0\sw\alpha\sw\del_1$.

Thus $s=0$ and $Q=P+\alpha$. So $Q[1\ton]$ intersects $P[1\ton]$, and also $Q[1\ton]$ does not intersect $P[1\ton]+\del_1$ (because $x\notin \F(f)$). Thus $Q[1\ton]$ intersects $P[1\ton]$ twice. Assume first that $Q[1]$ is below $P$. Let $Q':=Q+\del_2$.  Thus $Q'[0]\approx Q[1]$ and $Q'$ is above $Q$. Therefore, $Q'$ must intersect $P[1\ton]$. Since $P=Q'+x-\del$, we see that $x\in \F(f)$, a contradiction. If $Q[1]$ is above $P$ then we get a contradiction by \cref{lemma:impossible}.
\end{proof}

Let $\Del_1$ be the convex hull of $\{0,\del_1,\del\}$ and $\Del_2$ be the convex hull of $\{0,\del_2,\del\}$. Denote $\PZ:=(\Del_1\cup \Del_2)\cap \Z^2$.

\begin{lemma}\label{lemma:triangle}
We have $\PZ\subset \Fb(f)$.
\end{lemma}
\begin{proof}
Let $x\in\Z^2\cap \Del_1$, and suppose that $x\notin \Fb(f)$. First, $x$ cannot be northwest of $\del_1$ by Lemma~\directref{lemma:F_properties_easy}{item:order_ideal}.  Thus either $x\sw\del_1$ or $\del_1\sw x$.

Assume first that $x\sw \del_1$. Then for $\alpha:=\del_1-x$, we have $\slope(\alpha)\leq \slope(\del_1)$ and $\<\delp,\alpha\><0$. If $\alpha\notin\F(f)$ then by Lemma~\directref{lemma:F_properties_easy}{item:semigroup}, we get $\del_1\notin\F(f)$, a contradiction. Thus $\alpha\in \F(f)$, in which case we are done by \cref{lemma:del_1-alpha}.

Assume now that $\del_1\sw x$. Applying the dual argument (cf. \cref{prop:rotn}) to $y:=\del-x$, we find $y\in\F(f)$, and thus $x\in\F(f)$ by \cref{cor:cs}. Thus $\Fb(f)$ contains all lattice points of $\Del_1$, and by \cref{cor:cs} again, it contains all lattice points of $\Del_2$ as well.
\end{proof}

\begin{proposition}\label{prop:F=G}
We have $\F(f_1)=G_1$ and $\F(f_2)=G_2$, where
\begin{align*}%
  G_1'&:=\{\alpha\mid \alpha\in \F(f)\setminus\{\del_1\}\text{ is such that $\slope(\alpha)\leq\slope(\del_1)$}\},\\
  G_2'&:=\{\beta\mid \beta\in \F(f)\setminus\{\del_2\}\text{ is such that $\slope(\beta)\geq\slope(\del_2)$}\},\\
G_1&:=G_1'\cup (\del_1-G_1'), \quad\text{and}\quad G_2:=G_2'\cup (\del_2-G_2').
\end{align*}
\end{proposition}
\noindent See \cref{fig:recurrence} for an example.
\begin{proof}
By \cref{prop:rotn}, it suffices to prove $\F(f_1)=G_1$. By \cref{cor:cs}, $\F(f_1)$ is symmetric with respect to the map $\alpha\mapsto \del_1-\alpha$, so by \cref{cor:corner_slope_stong}, $G_1\subset \F(f_1)$. Conversely, suppose that we have found $\alpha\in \F(f_1)\setminus G_1$. Since both sets are symmetric with respect to the map $\alpha\mapsto \del_1-\alpha$, we may assume that $\slope(\alpha)\leq \slope(\del_1)$. By the definition of $G_1$, we have $\alpha\notin \F(f)$, so $\slope(\alpha)< \slope(\del_1)$. By \cref{slope:<}, $Q:=P+\alpha$ is below $\Pinf$. 

Since $\alpha\in \F(f_1)$, $Q[1\ton]$ intersects $(P[1\ton])_\infty$ twice, and both intersections must belong to $P[1\ton]\cup P'[1\ton]$, where $P':=P+\del_1$. Since $\alpha\notin \F(f)$, we see that $Q[1\ton]$ cannot intersect $P[1\ton]$, so it intersects $P'[1\ton]$ twice. Observe that $P'[1]\approx P[n]$ is above $\Qinf$. We get a contradiction by \cref{lemma:impossible} (applied to $\alpha':=\del_1-\alpha$, $Q$, and $P'$).
\end{proof}

\begin{corollary}\label{cor:supset_12}
We have
\begin{equation}\label{eq:F=big-cup}
\Fb(f)= \PZ\cup  \F(f_1)\cup \F(f_2)\cup(\F(f_1)+\del_2)\cup(\F(f_2)+\del_1).
\end{equation}
\end{corollary}
\begin{proof}
By \cref{rmk:semigroup_strong}, $\F(f)$ contains no points which are southeast of $\del_1$ or northwest of $\del_2$. By \cref{lemma:triangle}, $\Fb(f)$ contains $\PZ$. For any $\alpha\in [k-1]\times[n-1]$ satisfying $0\sw\alpha\sw \del_1$ and $\slp(\alpha)\leq \slp(\del_1)$, we have $\alpha\in\F(f)$ if and only if $\alpha\in \F(f_1)$ by \cref{prop:F=G}. The case $0\sw \alpha\sw \del_2$ and $\slp(\alpha)\geq \slp(\del_2)$ is handled similarly. The remaining two cases follow from the observation that both sides of~\eqref{eq:F=big-cup} are centrally symmetric.
% Let $\alpha\in [k-1]\times[n-1]$ be such that $\slp(\alpha)\leq \slp(\del)$. We need to show that $\alpha\in\F(f)$ if and only if $\alpha$ belongs to the right hand side of~\eqref{eq:F=big-cup}, and it suffices to consider the cases $ 
 % By the definition of $G_1$, $\F(f)$ contains all points $\alpha\in \F(f_1)$ satisfying $\slope(\alpha)\leq \slope(\del_1)$. For each such point, $\F(f)$ contains $\del_1-\alpha$ by \cref{lemma:del_1-alpha}. Thus $\F(f_1)\subset \F(f)$. The other three inclusions are proven similarly.
\end{proof}

\begin{lemma}\label{lemma:slopes_12}
For all $\alpha\in \F(f_1)$ and $\beta\in \F(f_2)$, we have $\slope(\alpha)<\slope(\beta)$. 
\end{lemma}
\begin{proof}
Assume otherwise that $\slope(\alpha)\geq \slope(\beta)$ for some $\alpha\in \F(f_1)$ and $\beta\in \F(f_2)$. We will consider the cases $\slp(\alpha)=\slp(\beta)$ and $\slp(\alpha)>\slp(\beta)$ separately.

Suppose that $\slp(\alpha)=\slp(\beta)$. Denote $\alpha=(a,b)$ and let $x:=\frac1{\gcd(a,b)} \alpha$. Thus each of $\alpha$ and $\beta$ is a positive integer multiple of $x$. By Lemma~\directref{lemma:F_properties_easy}{item:semigroup}, we have $x\in \F(f_1)\cap \F(f_2)$. Suppose first that $\slp(x)\geq \slp(\del_1)$. Let $y:=\del_1-x$, thus $\slp(y)\leq \slp(\del_1)$ and $y\sw \del_1$. By \cref{lemma:southwest}, $\del-y\notin \Fb(f_2)+\Z\del_2$. On the other hand, $\del-y=\del_2+x$ which clearly belongs to $\Fb(f_2)+\Z\del_2$ since $x\in \F(f_2)$, a contradiction. Applying a dual argument (cf. \cref{prop:rotn}) yields a contradiction when $\slp(x)\leq \slp(\del_2)$. Since $\slp(\del_1)<\slp(\del_2)$, we are done with the case $\slp(\alpha)=\slp(\beta)$.

Suppose now that $\slp(\alpha)>\slp(\beta)$. By \cref{lemma:inf_slopes}, we find that for some $s,t, \in \Z$, we have that $P[1\ton]+s\alpha$ crosses $P[0\to1]+t\beta$ from below, contradicting \cref{prop:cross_below}. 
\end{proof}
\begin{corollary}\label{cor:vert_12}
The points $\del_1$ and $\del_2$ are vertices of the convex hull of $\Fbf$.
\end{corollary}
\begin{proof}
Indeed, let $\alpha\in \F(f_1)$ have the maximal slope and $\beta\in \F(f_2)$ have the minimal slope. Then the convex hull of $\Fbf$ is bounded from below by the rays $\del_1-\R_{\geq0}\alpha$ and $\del_1+\R_{\geq0}\beta$ and from above by the rays $\del_2-\R_{\geq0}\beta$ and $\del_2+\R_{\geq0}\alpha$.
\end{proof}

\begin{proof}[Proof of \cref{thm:convex}]
We proceed by induction on $k$ and $n$. Suppose that the statement is known for all smaller $k$ and $n$, and consider some lattice point $x\notin \Fbf$ which belongs to the convex hull of $\Fbf$. By \cref{lemma:triangle}, we have $x\notin\PZ$. By \cref{cor:vert_12}, $\del_1$ and $\del_2$ are vertices of the convex hull of $\Fbf$. By the induction hypothesis, we know that the sets $\F(f_1)$ and $\F(f_2)$ are convex.  We obtain a contradiction with \cref{cor:supset_12}, so we must have $x \in \Fbf$.
\end{proof}

\section{Concave profiles and the counting formula}
By \cref{cor:cs} and \cref{thm:convex}, if $f\in\Bknc$ is \repfree then $\Fr(f)$ is convex and centrally symmetric. In this section, we show that each convex centrally symmetric set arises in this way, as stated in \cref{thm:main:existence}. We will use this construction to prove the counting formula~\eqref{eq:counting} in \cref{sec:counting}, completing the proof of \cref{thm:main}.

\subsection{Concave profiles}\label{sec:concave}
\begin{definition}\label{dfn:Hbf}
A sequence $\Hbf:=(0=\H_0,\H_1,\dots,\H_n=k)$ of real numbers is called a \emph{concave profile} if 
\begin{itemize}
\item $0<\H_{i+1}-\H_i<1$ for all $0\leq i<n$,
\item $\H_{i+1}-\H_i\geq\H_{j+1}-\H_j$ for all $0\leq i\leq j<n$, and
\item $\h_i\neq \h_j$ for $0\leq i\neq j<n$, where we set 
\begin{equation}\label{eq:H_to_h}
  \h_r:=\H_r-\lf \H_r\rf \quad\text{for}\quad 0\leq r\leq n.
\end{equation}
\end{itemize}
\end{definition}
Given a concave profile $\Hbf$, we let %
\begin{equation*}%
  \F(\Hbf):=\{(a,b)\in[k-1]\times[n-1]\mid  k-\H_{n-b}\leq a\leq \H_b\}.
\end{equation*}
As before, we let $\Fb(\Hbf):=\F(\Hbf)\sqcup\{(0,0),(k,n)\}$. We also let $\PHbf$ be the path connecting the points $(r,\H_r)$ for $r=0,1,\dots,n$. Thus $\Fb(\Hbf)$ consists of all lattice points weakly below $\PHbf$ and weakly above the $180^\circ$-rotation $(k,n)-\PHbf$ of $\PHbf$.

\begin{figure}
  \includegraphics[width=1.0\textwidth]{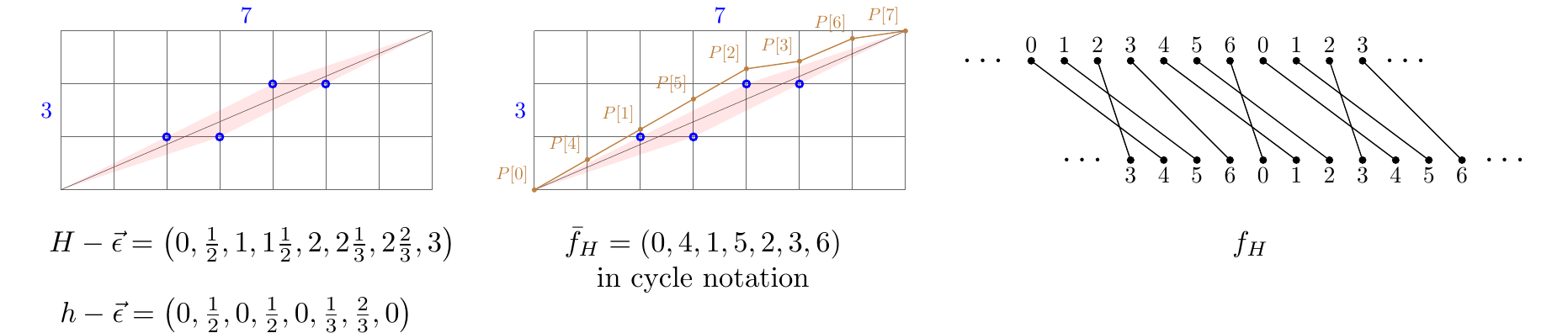}
  \caption{\label{fig:profile} Constructing a concave profile $\Hbf$ and a \repfree permutation $f_\Hbf$ for a given convex set $\Flet'$. See \cref{prop:constructing_profile} and \cref{dfn:f_Hbf}.}
\end{figure}

Denote
\begin{equation*}%
  \Fmax:=\{(a,a+b)\mid (a,b)\in[k-1]\times[n-k-1]\}.
\end{equation*}

\begin{proposition} \label{prop:constructing_profile}
Let $\Flet'\subset\Fmax$ be convex and centrally symmetric.  Then there exists a concave profile $\Hbf$ satisfying $\Flet'=\F(\Hbf)$.
\end{proposition}
\begin{proof}
Choose a nonnegative strictly concave sequence $\beps=(\eps_0,\eps_1,\dots,\eps_n)$ whose values are sufficiently small, and let $\Hbf$ be such that the difference $\Hbf-\beps$ records the maximal vertical coordinates of the intersection of the convex hull of $\Fbi'$ with the vertical line $\n(x)=i$ for $i=0,1,\dots,n$. Then clearly $\Hbf$ is a concave profile and we have $\Flet'=\F(\Hbf)$. See \cref{fig:profile} for an example.
\end{proof}

 The following construction uses $\Hbf$ to find a bounded affine permutation $f_\Hbf\in\Bknc$ satisfying the desired properties.
\begin{definition}\label{dfn:f_Hbf}
Given a concave profile $\Hbf$, let $f=f_\Hbf\in\Bknc$ be the unique bounded affine permutation such that for all $0\leq i,j<n$, we have $\fb^i(0)<\fb^j(0)$ if and only $\h_i<\h_j$, where $\h_i,\h_j$ are defined in~\eqref{eq:H_to_h}. In other words, writing $\fb=(0,j_1,j_2,\dots,j_{n-1})$ in cycle notation, the indices $(j_1,j_2,\dots,j_{n-1})$ have the same relative order as $(h_1,h_2,\dots,h_{n-1})$. See \cref{fig:profile} for an example.
\end{definition}

\begin{proposition}\label{prop:Hbf}
Let $\Hbf$ be a concave profile and $f:=f_\Hbf$. We have:
\begin{theoremlist}
\item\label{item:Hbf1} $\left\lf \frac{f^r(0)}n\right\rf =\left\lf \H_r\right\rf$ for all $0\leq r\leq n$;
\item\label{item:Hbf2} $f\in\Bknc$ is \repfree;
\item\label{item:Hbf3} $\F(f)=\F(\Hbf)$.
\end{theoremlist}
\end{proposition}
\begin{proof}
\itemref{item:Hbf1}: We prove the result by induction on $r$. The base case $r=0$ is clear. Suppose that the result holds for $0\leq r<n$. We have $\h_{r+1}\neq\h_r$. If $\h_{r+1}>\h_r$ then $\fb^{r+1}(0)>\fb^{r}(0)$, and thus $\left\lf \frac{f^{r+1}(0)}n\right\rf=\left\lf \frac{f^r(0)}n\right\rf$. It is also clear that $\h_{r+1}>\h_r$ implies $\lf\H_{r+1}\rf=\lf H_r\rf$. Similarly, if $\h_{r+1}<\h_r$ then $\fb^{r+1}(0)<\fb^r(0)$, which implies $\left\lf \frac{f^{r+1}(0)}n\right\rf=\left\lf \frac{f^r(0)}n\right\rf+1$ and $\lf\H_{r+1}\rf=\lf H_r\rf+1$.

\itemref{item:Hbf2}: Let $\PHbfinf:=\bigcup_{t\in\Z}(\PHbf+t\del)$ be the corresponding infinite path. Observe that for each $\alpha=(a,b)\in\Z^2$, we have $|\PHbfinf\cap (\PHbfinf+\alpha)|=|\Pfinf\cap(\Pfinf+\alpha)|$.  Indeed, if $p\in \PHbf$ and $q\in \PHbf+\alpha$ have the same horizontal coordinate $r\in\Z$ then $p$ is above $q$ if and only if $\H_r>a+\H_{r-b}$. This condition is equivalent to having either $\lf \H_r\rf >a+\lf \H_{r-b}\rf$ or $\lf \H_r\rf =a+\lf \H_{r-b}\rf$ and $\h_r>\h_{r-b}$. By~\itemref{item:Hbf1}, this is equivalent to having $\frac{f^r(0)}n>a+\frac{f^{r-b}(0)}n$, which means that for the integer points $p'\in \Pf$ and $q'\in \Pf+\alpha$ satisfying $\n(p')=\n(q')=r$, the point $p'$ is above $q'$. Since the path $\PHbf$ is the plot of a concave function, it intersects $\PHbf+\alpha$ at most once for each $\alpha\in\Z^2$. Thus $\PHbfinf$ intersects $\PHbfinf+\alpha$ at most twice, and therefore the same holds for $\Pfinf$. The result follows by \cref{prop:F_paths}.

\itemref{item:Hbf3}: For $\alpha\in[k-1]\times[n-1]$, $\PHbfinf$ intersects $\PHbfinf+\alpha$ if and only if $\alpha$ is below $\PHbf$ and $(k,n)$ is below $\PHbf+\alpha$. This is equivalent to $\alpha\in \F(\Hbf)$. Since $|\PHbfinf\cap (\PHbfinf+\alpha)|=|\Pfinf\cap(\Pfinf+\alpha)|$, this is equivalent to $\alpha\in\F(f)$.
\end{proof}

\subsection{Counting formula for concave profiles}\label{sec:counting}
We prove~\eqref{eq:counting} in two steps. We start by treating the case where $f=f_\Hbf$ arises from a concave profile. The case of arbitrary \repfree $f\in\Bknc$ is considered in \cref{sec:proof} below.

\begin{proposition}\label{prop:counting_concave}
Let $\Hbf$ be a concave profile and let $f:=f_\Hbf$. Then 
\begin{equation*}%
  \Cat_f=\#\DyckF(f).  
\end{equation*}
\end{proposition}
\begin{proof}
Let us say that a \emph{slanted Dyck path} is a lattice path connecting $(0,0)$ to $(k,n)$ which stays above the main diagonal and consists of right steps $(0,1)$ and up-right steps $(1,1)$. Thus $\#\DyckF(f)$ counts the number of slanted Dyck paths which stay above $\PHbf$ (and do not share any points with $\PHbf$ except for the endpoints $(0,0)$ and $(k,n)$).

In order to keep track of the size of the rectangle in which $\Hbf$ lives, let us refer to $\Hbf$ as a \emph{$(k,n)$-concave profile}. We proceed by induction on $n$ using \cref{prop:Cat_rec}. The base case $n=1$ is clear. Suppose now that $n>1$ and that the claim has been shown for all $n'<n$ and also for all $(k,n)$-concave profiles $\Hbf'$ satisfying $\F(\Hbf')\supsetneq \F(\Hbf)$. Let $0<r<n$ be the index such that $0<\h_r<1$ is maximal among $\h_0,\h_1,\dots,\h_n$. Thus $\fb^r(0)=n-1$, and we let $\eps:=1-\h_r$.

Assume first that $r=1$. Let $g\in \Bxc_{k-1,n-1}$ be given by $\gb^i(0):=\fb^{i+1}(0)$ for $0<i< n$, and let $\Hbf':=(0=\H'_0,\H'_1,\dots,\H'_{n-1}=k-1)$ be given by $\H'_i:=\H_{i+1}+\eps-1$ for $0\leq i<n-1$. It is easy to check that $g=f_{\Hbf'}$ and that removing the first step (which must be up-right) of a slanted Dyck path above $\PHbf$ yields a slanted Dyck path above $\PHbfp$ and vice versa, thus $\#\DyckF(f)=\#\DyckF(g)$. Applying parts~\itemref{item:Cat_MS_fixedpt}--\itemref{item:Cat_MS_i_i+1} of \cref{prop:Cat_rec}, we find $\Cat_f=\Cat_g$. 

Assume next that $r=n-1$. Let $g\in \Bxc_{k,n-1}$ be given by $\gb^i(0):=\fb^{i}(0)$ for $0\leq i<n-1$, and let $\Hbf':=(0=\H'_0,\H'_1,\dots,\H'_{n-1}=k)$ be given by $\H'_i:=\H_{i}+\eps$ for $0< i<n$. Similarly to the above, we have $\#\DyckF(f)=\#\DyckF(g)$ and $\Cat_f=\Cat_g$.

Finally, assume that $1<r<n-1$. Let $i:=-1$, $j=0$, $a:=f^{-1}(i)$, $b:=f^{-1}(j)$, $c:=f(i)$, $d:=f(j)$. By \cref{dfn:Hbf}, we have $\H_1-\H_0\geq \H_{r+1}-\H_r$, and since $h_r$ is maximal, we get $\h_1\geq 1+\h_{r+1}-\h_r>\h_{r+1}$. Similarly, using $\H_n-\H_{n-1}\leq \H_{r}-\H_{r-1}$ we get $\h_{n-1}>\h_{r-1}$. By \cref{dfn:f_Hbf}, this implies $d>c$ and $b>a$. We thus have $a<b<i<j<c<d$.

\begin{figure}
  \includegraphics{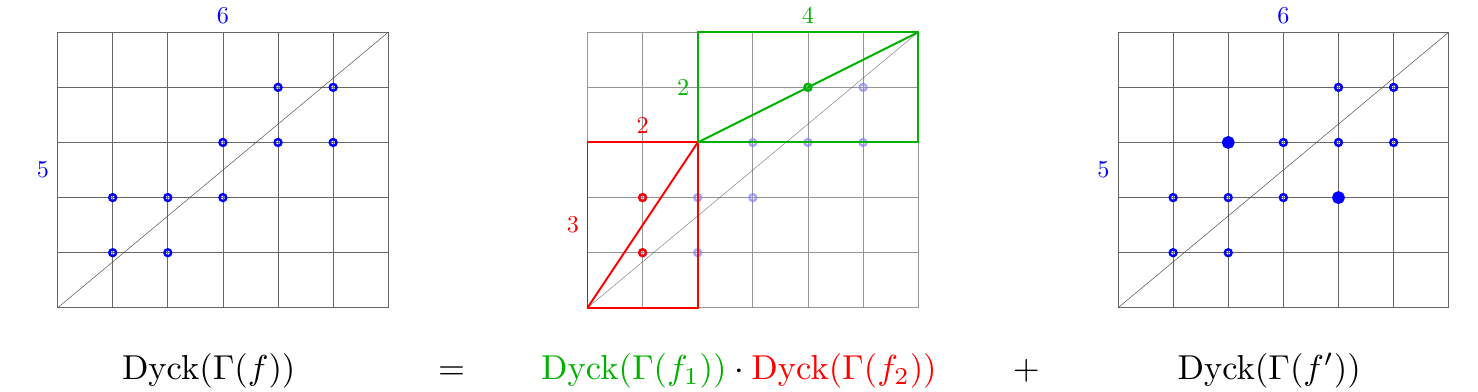}
  \caption{\label{fig:recurrence} The Dyck path recurrence in the proofs of \cref{prop:counting_concave} and \cref{thm:main:convex_counting}. Here $\Fr(f')=\Fr(f)\sqcup\{(a,b),(n-k-a,k-b)\}$ for some $a,b$. Each Dyck path above $\Fr(f)$ either passes through $(a,b)$ or stays above $\Fr(f')$.}
\end{figure}

Let $f':=s_ifs_i$. Our goal is to relate $\Cat_{f'}$ to $\Cat_{f_1}$, $\Cat_{f_2}$, and $\Cat_f$ as shown in \cref{fig:recurrence}. 
 It follows that $f'\in\Bknc$ and that $f'$ has a double crossing at $i$. Let $f_1,f_2$ be obtained from $f$ by resolving the crossing $(i,i+1)$. By \cref{cor:Cat_rec}, we have 
\begin{equation}\label{eq:Cat_rec_proof}
    \Cat_{f}=\Cat_{f_1}\Cat_{f_2}+\Cat_{f'}.
\end{equation}
Let $g = \sigma f' \in\Bknc$ be the cyclic shift of $f'$ defined in~\eqref{eq:cyc_shift}. We have $f^r(0)\equiv n-1$ and $g^r(0)\equiv 1$ modulo $n$, and for $1\leq s\leq n$ such that $s\neq r$, we have $g^s(0)=f^s(0)+1$. Choose $\eps'>\eps$ such that $\eps'<1-h_s$ for $s\neq r$ and let $\Hbf':=(0=\H'_0,\H'_1,\dots,\H'_n=k)$ be given by $\H'_s:=\H_s+\eps'$ for all $0<s<n$. One easily checks that $g=f_{\Hbf'}$. Since $\F(g)\supsetneq \F(f)$, by the induction hypothesis, we have $\Cat_g=\#\DyckF(g)$. By \cref{prop:cyc_shift_iso}, we have $\Cat_{f'}=\Cat_g$ and $\F(f')=\F(g)$, thus $\Cat_{f'}=\#\DyckF(f')$. It is straightforward to check that there exist concave profiles $\Hbf^{(1)}$ and $\Hbf^{(2)}$ such that $f_1=f_{\Hbf^{(1)}}$ and $f_2=f_{\Hbf^{(2)}}$, thus by the induction hypothesis, \eqref{eq:Cat_rec_proof} becomes
\begin{equation}\label{eq:Cat_rec_proof2}
\Cat_f=\#\DyckF(f_1)\cdot \#\DyckF(f_2)+\#\DyckF(f').
\end{equation}
On the other hand, it is clear that $\F(f')= \F(f)\sqcup \{(k_2,r),(k-k_2,n-r)\}$, where $k_2:=\frac{f^r(0)+1}n$. The number of slanted Dyck paths above $\PHbf$ passing through the point $(k_2,r)$ equals $\#\DyckF(f_1)\cdot \#\DyckF(f_2)$. The slanted Dyck paths above $\PHbf$ which do not pass through $(k_2,r)$ must stay above $\PHbfp$. Therefore $\Cat_f=\#\DyckF(f)$.
\end{proof}

\subsection{Finishing the proof of Theorem~\ref{thm:main}}\label{sec:proof}
\begin{proof}[Proof of \cref{thm:main:convex_counting}]
Let $f\in\Bknc$ be \repfree. By \cref{cor:cs}, $\Fr(f)$ is centrally symmetric. By \cref{thm:convex}, $\Fr(f)$ is convex. It remains to show the counting formula~\eqref{eq:counting}. Recall from \cref{prop:counting_concave} that the formula holds when $f=f_\Hbf$ arises from a concave profile. In particular, we may choose $\Hbf$ to be such that $\F(\Hbf)=\Fmin$. We now proceed by induction. By \cref{Fmin_MSeq}, the counting formula extends to all \repfree $f\in\Bknc$ satisfying $\F(f)=\Fmin$, which is the base case. Suppose now that $\F(f)\supsetneq \Fmin$ and that the result has been shown for all $n'<n$ and for all \repfree $f'\in\Bknc$ such that $\F(f')\subsetneq\F(f)$. (This induction proceeds in the opposite direction to the one in the proof of \cref{prop:counting_concave}.) By \cref{find_double_cross}, after applying some length-preserving simple conjugations, we may assume that $f$ has a double crossing at some $i\in\Z$. 
 Let $f':=s_ifs_i$ (thus $\ell(f)=\ell(f')+2$) and $f_1,f_2$ be obtained from $f$ by resolving the crossing $(i,i+1)$. By \cref{cor:Cat_rec}, we have 
\begin{equation}\label{eq:Cat_rec_proof3}
    \Cat_{f'}=\Cat_{f_1}\Cat_{f_2}+\Cat_{f}.
\end{equation}
This is different from~\eqref{eq:Cat_rec_proof} in that $f$ and $f'$ are swapped. By induction, we have $\Cat_{f_1}=\#\DyckF(f_1)$, $\Cat_{f_2}=\#\DyckF(f_2)$, and $\Cat_{f'}=\#\DyckF(f')$. Similarly to the proof of \cref{prop:counting_concave} (cf. \cref{fig:recurrence}), we obtain the desired result $\Cat_f=\#\DyckF(f)$.
\end{proof}

\begin{proof}[Proof of \cref{thm:main:existence}]
As explained in \cref{sec:concave}, for any convex centrally symmetric set $\Flet'\subset\Fmax$, there exists a concave profile $\Hbf$ such that $\F(\Hbf)=\Flet'$. The result follows from \cref{prop:Hbf}.
\end{proof}

\section{Other interpretations and further directions} \label{sec:other}
Computer experimentation reveals many other remarkable properties of \repfree bounded affine permutations which we state below in conjectural form. We discuss them from a knot-theoretic perspective and in relation to positroid varieties, motivated by our recent results~\cite{qtcat}. We also discuss the various interpretations of positroid Catalan numbers mentioned in the introduction.

\subsection{Euler characteristic of open positroid varieties}\label{sec:Euler}
The relation between~\cref{defn:intro} and~\eqref{eq:Catf} follows from \cite{qtcat}; here we give a brief explanation.  When $f \in \Bknc$, the torus $T$ acts freely on $\Pio_f$ and the quotient $\Pit_f:=\Pio_f/T$ is a smooth affine variety, called the \emph{positroid configuration space} in \cite{qtcat}.  The torus-equivariant Euler characteristic of $\Pio_f$ is simply the usual Euler characteristic of $\Pit_f$.  The point count is given by $\#(\Pit_f)(\Fbb_q) = \Rt_f(q)$.  
By the Grothendieck--Lefschetz fixed-point formula, when a smooth variety $X$ has polynomial point count, its Euler characteristic is equal to $\#X(\Fbb_q)|_{q=1}$.  This shows the agreement of~\cref{defn:intro} and~\eqref{eq:Catf}.

When $f \in \Bkn \setminus \Bknc$, the torus $T$ no longer acts freely on $\Pio_f$, and the torus-equivariant cohomology $H^*_{T}(\Pio_f)$ (or compactly supported cohomology $H^*_{T,c}(\Pio_f)$) is typically infinite-dimensional.  In this case, \cref{defn:intro} does not immediately apply, but a $q,t$-power series is studied in \cite{qtcat}.  In the present work, we use~\eqref{eq:Catf} for all $f \in \Bkn$, but caution the reader that the situation is more subtle when $f \notin \Bknc$.

\subsection{Generalized $q,t$-Catalan numbers}\label{sec:gen_qtcat}

The most exciting computational evidence arises when comparing our results to the constructions in~\cite{GHSR,BHMPS}.

\begin{figure}
  \includegraphics{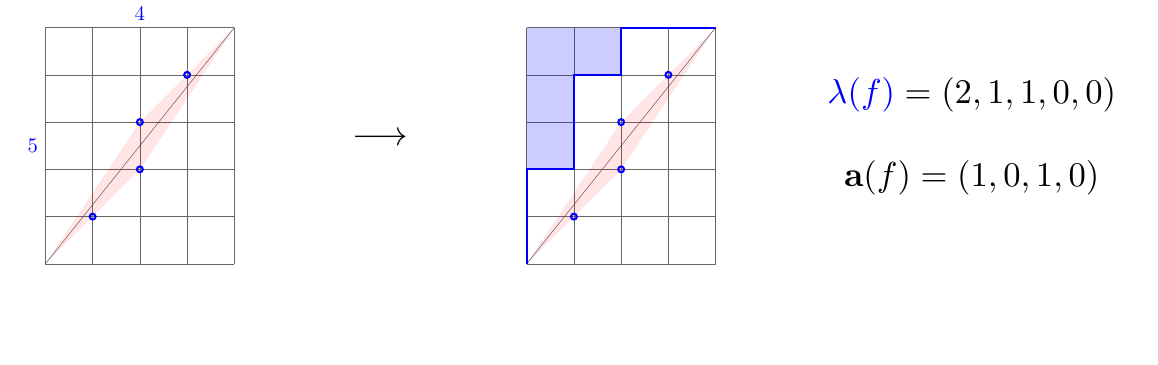}\vspace{-0.2in}
  \caption[]{\label{fig:lambda} Constructing a Young diagram $\la(f)$ and the sequence $\a(f)$ from $\Fr(f)$. Here $f\in\Bknc$ for $k=5$ and $n=9$.}
\end{figure}

Given a \repfree $f\in\Bknc$, let $\la(f)=(\la_1,\la_2,\dots,\la_k)$ be the partition consisting of all boxes inside the $k\times (n-k)$ rectangle which are above the diagonal and are strictly above all points of $\Fr(f)$; see \cref{fig:lambda}. Let $\a(f)=(a_2,\dots,a_{k})$ be given by $a_i=\la_{i-1}-\la_{i}$ for $2\leq i\leq k$. To an arbitrary sequence $\a=(a_2,\dots,a_k)$ of nonnegative integers, the authors of \cite{GHSR} associate a \emph{generalized $q,t$-Catalan number}\footnote{What we denote by $\Cata(q,t)$ was denoted by $F(a_2,\dots,a_k)$ in~\cite{GHSR}.} $\Cata(q,t)$, which may be explicitly described as a combinatorial sum over \emph{Tesler matrices}. According to~\cite[Conjecture~1.3]{GHSR} attributed to A.~Negu\c{t}, if $a_2\geq a_3\geq\cdots\geq a_k\geq0$ then $\Cata(q,t)$ has positive integer coefficients. If this condition is satisfied then we have $\a=\a(f)$ for some \repfree $f\in\Bknc$ in view of \cref{thm:main:existence}. However, the convexity condition is more general: for instance, the sequence $\a(f)=(1,0,1,0)$ in \cref{fig:lambda} is not weakly decreasing.

Each sequence $\a$ also gives rise to a \emph{Coxeter link} $\betah(\a)$. See~\cite{GN,GNR,OR_Cox,GHSR} and references therein for further details, such as an interpretation of $\Cata(q,t)$ in terms of flag Hilbert schemes and a conjectural relation between $\Cata(q,t)$ and Khovanov--Rozansky homology~\cite{KR1,KR2} of $\betah(\a)$.

In~\cite[Definition~1.9]{qtcat}, we have associated a knot $\betah_f$ to each $f\in\Bknc$ and we showed in~\cite[Theorem~1.11]{qtcat} that $\Rt_f(q)$ may be computed from the \emph{\FLY} polynomial of $\betah_f$. More generally, we gave a simple relation between the mixed Hodge polynomial $\Pcal(\Pit_f;q,t)$ and Khovanov--Rozansky homology of $\betah_f$ in~\cite[Equation~(1.25)]{qtcat}. 

\begin{conjecture}\label{conj:GHSR}
Let $f\in\Bknc$ be \repfree.%
\begin{theoremlist}
\item The knots $\betah(\a(f))$ and $\betah_f$ are isotopic.
\item\label{conj:GHSR2} Up to a monomial in $q$ and $t$, we have $\Pcal(\Pit_f;q,t)=\Cataf(q,t)$.
\item\label{conj:GHSR3} Up to a monomial in $q$, we have $\Rt_f(q)=\Cataf(q,t=1/q)$.
\item The polynomials $\Pcal(\Pit_f;q,t)$, $\Cataf(q,t)$, and $\Rt_f(q)$ have positive integer coefficients.
\end{theoremlist}
\end{conjecture}

\begin{remark}
Combining \cref{thm:main:convex_counting} with~\cite[Proposition~1.1]{GHSR}, we see that $\Rt_f(1)=\Cataf(1,1)$, in agreement with \cref{conj:GHSR3}.
\end{remark}

\Cref{conj:GHSR} becomes especially intriguing in view of~\cite[Section~7]{BHMPS}. Namely, to each sequence $\a=(a_2,\dots,a_k)$ of nonnegative integers, the authors of~\cite{BHMPS} associate a symmetric function $\omega(D_{\a}\cdot 1)$ and show that one of the coefficients in its Schur expansion equals $\Cata(q,t)$. They conjecture that when $\a$ is obtained from a Young diagram above a concave curve (that is, precisely when $\a=\a(f)$ for some \repfree $f\in\Bknc$)  then the function $\omega(D_{\a}\cdot 1)$ is $q,t$ Schur positive. The appearance of the convexity condition in both of these settings suggests that the whole symmetric function $\omega(D_{\a(f)}\cdot 1)$ may have an interpretation in terms of the geometry of $\Pio_f$, which may explain the Schur positivity phenomenon. 

A related promising direction would be to ``categorify'' the recurrence~\eqref{eq:MSrec:double_cross} to the level of Khovanov--Rozansky homology, in the spirit of~\cite{Hog,Mellit_torus,MeHog}. Conversely, it would be interesting to ``decategorify'' the \emph{categorified Young symmetrizers} of~\cite{Hog} and interpret them in the positroid language.
We hope to return to these questions in future work.

\subsection{\Ceqvce classes} Our experiments indicate that the combinatorics of \Ceqvce classes has very rigid structure. Some statements describing this structure were shown in \cref{sec:conj-double-move}. The following conjecture implies that the objects $\Cat_f$, $\Rt_f(q)$, $\Pcal(\Pio_f;q,t)$, and $\betah_f$ depend only on $\F(f)$ when $f\in\Bknc$ is \repfree.

\begin{conjecture}\label{conj:classes}
Let $\Flet\subset[k-1]\times[n-k-1]$ be centrally symmetric and convex. Then the set
\begin{equation}\label{eq:conj_eq_classes}
  \{f\in\Bknc\mid \F(f)=\Flet\}
\end{equation}
is a union of $\gcd(k,n)$-many \Ceqvce classes. They are cyclically permuted by the map $\sigma$ from~\eqref{eq:cyc_shift}.
\end{conjecture}
\begin{remark}
Let $\epskn$ be equal to $1/2$ if both $k$ and $n$ are even and to $0$ otherwise. 
 For $f\in\Bknc$, denote
\begin{equation*}%
  \difmod(f):=\<\delp,\Pfinf\>-\epskn,
\end{equation*}
cf.~\eqref{eq:delp_Pinf}. It is not hard to see that $\difmod(f)$ is always an integer, so we let $0\leq \difmodb(f)\leq d-1$ be obtained by taking $\difmod(f)$ modulo $d:=\gcd(k,n)$.

Let $f\in \Bknc$ be \repfree. Observe that if $g:=\shf$ is the cyclic shift of $f$ then $\difmod(g)-\difmod(f)=1$, however, if $f'\ceq f$ then $\difmodb(f')=\difmodb(f)$. We therefore see that the set~\eqref{eq:conj_eq_classes} contains at least $d$ distinct \Ceqvce classes, cyclically permuted by $\sigma$. The content of \cref{conj:classes} is that if $f,g\in\Bknc$ are \repfree and satisfy $(\Fr(f),\difmodb(f))=(\Fr(g),\difmodb(g))$ then $f\ceq g$.
\end{remark}

\begin{figure}
  \includegraphics[width=1.0\textwidth]{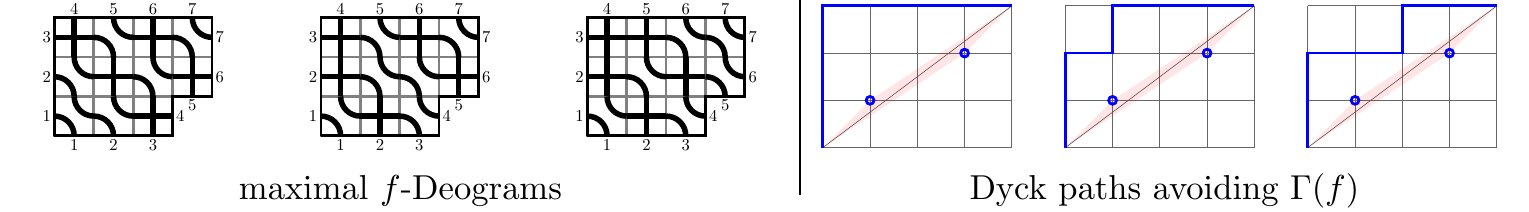}
  \caption[]{\label{fig:deograms} For each \repfree $f\in\Bknc$, the number of maximal $f$-Deograms (left) equals the number of Dyck paths which stay above $\Fr(f)$ (right); see \cref{ex:deograms} and \cref{prob:bij}.}
\end{figure}

\subsection{Deograms}\label{sec:deograms}
In~\cite[Section~9]{qtcat}, we explained that for each $f\in\Bknc$, the positroid Catalan number $\Cat_f$ equals the number $\#\Gom_f$ of certain combinatorial objects called \emph{maximal $f$-Deograms}, see~\cite[Definition~9.3]{qtcat}. Here, $\Gom_f$ denotes the set of maximal $f$-Deograms, defined as follows. First, by~\cite[Proposition~3.15]{KLS}, there exists a unique pair $(v,w)$ of permutations in $S_n$ such that $v\leq w$, $w^{-1}(1)<\dots<w^{-1}(k)$, $w^{-1}(k+1)<\dots<w^{-1}(n)$, and $\fb=wv^{-1}$. Thus $w$ is \emph{$k$-Grassmannian}, and each such permutation corresponds to a Young diagram $\la(w)$ that fits inside a $k\times (n-k)$ rectangle. An $f$-Deogram is obtained by placing a crossing~\crossing or an elbow~\elbow inside each box of $\la(w)$ so that (i) the resulting strand permutation is $v$, and (ii) a certain \emph{distinguished condition} is satisfied. An $f$-Deogram is \emph{maximal} if it contains the maximal possible number of crossings, equivalently, assuming $f\in\Bknc$, if it contains exactly $n-1$ elbows. In view of \cref{thm:main:convex_counting}, when $f$ is \repfree, $\Cat_f=\#\DyckF(f)$ also counts Dyck paths avoiding $\Fr(f)$.
\begin{example}\label{ex:deograms}
Let $\fb=(1,4,6,2,5,7,3)$ in cycle notation, and thus $f\in\Bknc$ for $k=3$ and $n=7$. The unique factorization $\fb=wv^{-1}$ as above is given by $v=\begin{pmatrix}
1&2&3&4&5&6&7\\
1&2&4&3&5&6&7
\end{pmatrix}$ and $w=\begin{pmatrix}
1&2&3&4&5&6&7\\
4&5&6&1&7&2&3
\end{pmatrix}$ in two-line notation. The three maximal $f$-Deograms are shown in \figref{fig:deograms}(left). The three Dyck paths avoiding $\Fr(f)$ are shown in \figref{fig:deograms}(right).
\end{example}

 The following problem extends~\cite[Problem~9.6]{qtcat}.
\begin{problem}\label{prob:bij}
Let $f\in\Bknc$ be \repfree. Find a bijection between $\Gom_f$ and $\DyckF(f)$. 
\end{problem}

\subsection{Fiedler invariant and knots in a thickened torus}
Let $\T^2:=\R^2/\Z^2$ be a torus and $K:S^1\hookrightarrow \T^2\times \R$ be a knot inside a thickened torus. To this data, Fiedler~\cite{Fiedler} associates an isotopy invariant $W_K$ called the \emph{small state sum}. Let us instead identify $\T^2$ with $\R^2/\<(0,n),(1,0)\>$. For $f\in\Bknc$ and $P:=\Pf$, let $\Pb$ be the image of $P\subset\R^2$ under the quotient map $\R^2\to \T^2$. The points where $\Pb$ intersects itself correspond precisely to the inversions of $f$. Thus we may define a knot $K_f$ inside $\T^2\times \R$ whose projection to $\T^2$ coincides with $\Pb$, and for each inversion $(i,j)$ of $f$, the line segment connecting $P[i]$ to $P[f(i)]$ lies above the line segment connecting $P[j]$ to $P[f(j)]$. See \cref{fig:knot}.

\begin{figure}
  \includegraphics[width=0.6\textwidth]{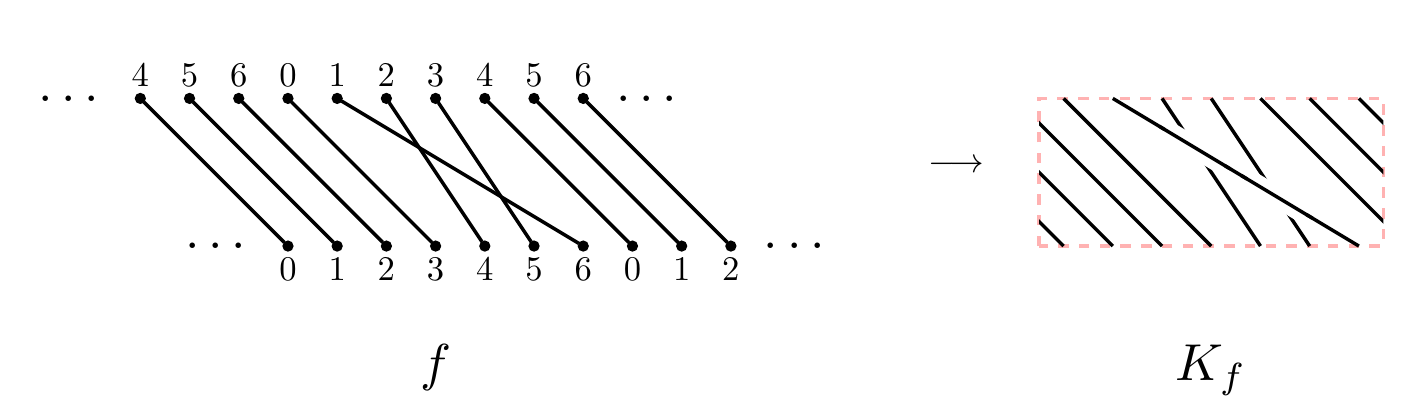}
  \caption{\label{fig:knot} Associating a knot $K_f$ inside $\T^2\times \R$ to $f\in\Bknc$. The dashed rectangle on the right represents the fundamental domain of $\T^2$.}
\end{figure}

It is straightforward to check that the formal sum $W_{K_f}$ contains essentially the same information as the inversion multiset $\F(f)$. This leads to the following question: which parts of our story generalize to arbitrary \emph{\repfree knots} inside $\T^2\times \R$? Here we say that a knot $K$ inside $\T^2\times \R$ is \repfree if each nonzero coefficient of $W_{K_f}$ is equal to $\pm1$. For example, it would be interesting to determine which subsets of $\Z^2/\Z\delta$ may appear with nonzero coefficients inside $W_{K}$ for a \repfree $K$, and whether the \FLY polynomial of $K$ (or its Khovanov--Rozansky homology) have nice properties when $K$ is \repfree.

\bibliographystyle{alpha}
\bibliography{cat_combin}

\end{document}